 \documentclass[11pt, reqno]{amsart}

\usepackage[utf8x]{inputenc}

\usepackage{amsmath, amsthm, amsfonts, amssymb}
\usepackage{epsfig, enumerate}
\usepackage{color}
\usepackage{dsfont}

\usepackage{pdf14}

\DeclareMathOperator{\rank}{{\rm rank}}

\newcommand{\trr}{\triangleleft}

\newtheorem{Theo}{Theorem}
\newtheorem{prop}{Proposition}
\newtheorem{lemma}[prop]{Lemma}

\newtheorem{defini}{Definition}

\newtheorem*{rem*}{Remark}
\newtheorem{rem}{Remark} 

\newcommand{\A}{{\mathbf A}}

\newcommand{\C}{{\mathbf C}}
\newcommand{\D}{{\mathbf D}}

\newcommand{\N}{{\mathbf N}}

\newcommand{\T}{{\mathbf T}}

\newcommand{\V}{{\mathbf V}}

\newcommand{\bfa}{{\mathbf a}}

\newcommand{\bfu}{{\mathbf u}}
\newcommand{\bfv}{{\mathbf v}}
\newcommand{\bfw}{{\mathbf w}}

\newcommand{\Aa}{{\mathcal A}}

\newcommand{\Cc}{{\mathcal C}}

\newcommand{\Ll}{{\mathcal L}}
\newcommand{\Mm}{{\mathcal M}}

\newcommand{\Rr}{{\mathcal R}}
\newcommand{\Ss}{{\mathcal S}}

\newcommand{\Uu}{{\mathcal U}}
\newcommand{\Vv}{{\mathcal V}}

\newcommand{\Ab}{{\mathbb A}}
\newcommand{\BB}{{\mathbb B}}
\newcommand{\CC}{{\mathbb C}}
\newcommand{\DD}{{\mathbb D}}
\newcommand{\EE}{{\mathbb E}}

\newcommand{\GG}{{\mathbb G}}
\newcommand{\HH}{{\mathbb H}}
\newcommand{\II}{{\one}}

\newcommand{\KK}{{\mathbb K}}

\newcommand{\NN}{{\mathbb N}}

\newcommand{\PP}{{\mathbb P}}

\newcommand{\RR}{{\mathbb R}}
\newcommand{\Sb}{{\mathbb S}}
\newcommand{\TT}{{\mathbb T}}
\newcommand{\UU}{{\mathbb U}}

\newcommand{\WW}{{\mathbb W}}

\newcommand{\ZZ}{{\mathbb Z}}

\newcommand{\BBb}{{\mathfrak B}}

\newcommand{\bbb}{{\mathfrak b}}

\newcommand{\one}{{\bf 1}}
\newcommand{\nul}{{\bf 0}}

\newcommand{\qtx}[1]{\quad\text{#1}\quad}

\newcommand{\GL}{{\rm GL}}

\newcommand{\Her}{{\rm Her}}

\newcommand{\pmat}[1]{\begin{pmatrix} #1  \end{pmatrix}}
\newcommand{\smat}[1]{\left( \begin{smallmatrix} #1  \end{smallmatrix} \right)}
\newcommand{\mat}[1]{\begin{matrix} #1  \end{matrix}}

\newcommand{\diag}{{\rm diag}}
\DeclareMathOperator{\im}{{\rm Im}}

\DeclareMathOperator{\supp}{{\rm supp}}
\DeclareMathOperator{\spec}{{\rm spec}}

\newcommand{\be}[1]{\begin{equation} \label{#1} }

\renewcommand{\CC}{\mathds{C}}
\renewcommand{\RR}{\mathds{R}}
\renewcommand{\ZZ}{\mathds{Z}}
\renewcommand{\NN}{\mathds{N}}

\numberwithin{prop}{section}
\numberwithin{equation}{section}

\newcommand{\Ups}{\Upsilon}

\newcommand{\HSP}{\HH\Sb\PP}
\newcommand{\SP}{\Sb\PP}

\usepackage{hyperref}

\title[Transfer matrices for discrete Hermitian operators]{
Transfer matrices for discrete Hermitian operators and absolutely continuous spectrum}

\author{Christian Sadel}
\address[Sadel]{Facultad de Matem\'aitcas, Pontificia Universidad Cat\'olica de Chile} 
\email{chsadel@mat.uc.cl}


\begin{document}

\begin{abstract}
We introduce a transfer matrix method for the spectral analysis of discrete Hermitian operators with locally finite hopping. Such operators can be associated with a locally finite graph structure and the method works in principle on any such graph. The key result is a spectral averaging formula well known for Jacobi or 1-channel operators giving the spectral measure at a root vector by a weak limit of products of transfer matrices. 
Here, we assume an increase in the rank for the connections between spherical shells which is a typical situation and true on finite dimensional lattices $\ZZ^d$.
The product of transfer matrices are considered as a transformation of the relations of 'boundary resolvent data' along the shells. The trade off is that at each level or shell with more forward then backward connections (rank-increase) we have a set of transfer matrices at a fixed spectral parameter. Still, considering these products  we can relate the minimal norm growth over the set of all products with the spectral measure at the root and obtain several criteria for absolutely continuous spectrum. 
Finally, we give some example of operators on stair-like graphs (increasing width) which has absolutely continuous spectrum with a sufficiently fast decaying random shell-matrix-potential.
\end{abstract}

\maketitle

\tableofcontents


\section{Introduction}

\subsection{The basic model}

Let $\GG$ be some countable set and let $H$ be a (possibly unbounded) symmetric operator on $\ell^2(\GG)$ with locally finite hopping. 
Locally finite hopping means that for any $x\in \GG$  the set $\{y\in\GG:\langle\delta_x, H \delta_y\rangle\neq 0 \}$ is finite where $\langle\cdot , \cdot\rangle$ is the scalar product of $\ell^2(\GG)$ 
(which will be anti-linear in the first and linear in the second component here) and $\delta_x$ is the unit vector in $\ell^2(\GG)$ with $\delta_x(x)=1$ and $\delta_x(y)=0$ for $y\neq x$. 
A large group of discrete models fall into this category, in fact any finite range hopping Schrödinger operator on any locally finite graph like $\ell^2(\ZZ^d)$ is of this type.
Because of the locally finite hopping condition, $H$ extends naturally to an operator on the set of all functions from $\GG$ to $\CC$, $\CC^\GG$. 
The natural and maximal domain of $H$ is then given by all functions $\psi\in\ell^2(\GG)\subset \CC^\GG$ such that $H\psi\in \ell^2(\GG)$. 
For the spectral theory we will later assume that $H$ is self-adjoint with this maximal domain.

Part of the motivation for this work was the development of criteria for proving existence of absolutely continuous spectrum in higher dimensional models.
A big open problem is the so called extended states conjecture for the Anderson model in $\ell^2(\ZZ^d)$ with small disorder in $d\geq 3$ dimensions. The fact that this important problem in Mathematical Physics is 
unsolved for decades emphasizes the need for developing such methods.
The Anderson model on a general graph $\GG$ is given by a random operator on $\ell^2(\GG)$ which is the sum of a graph -Laplacian or adjacency operator and a random potential which is  independent identically distributed at each site.
The general wisdom (but it may depend on the graph) is that for large disorder (large variance) and at the edges of the spectrum the random potential dominates the spectral structure and one has pure point spectrum and so called Anderson localization. There are two general methods to prove this, the fractional moment method \cite{AM} and multi-scale analysis \cite{FS, GK1, GK2}.
The fractional moment method at high disorder works fine in graphs with a finite upper bound on the connectivity of one point \cite{Tau}. 
For a long time the high disorder Bernoulli Anderson model on $\ZZ^d$ for $d>1$ could not be handled for a long time. However, recently the localization has also been shown in this case in $2$ and $3$ dimensions \cite{Li,LZ}.
Existence of continuous spectrum for Anderson models at low disorder has first been proved for infinite dimensional hyperbolic type graphs like regular trees and tree-like structures \cite{Kl, ASW, AW, FHS, FHS2, KS, KLW, Sa-FC, Sa-Fib}.
Using averaging effects for the potential, it has also been shown on special graphs with a finite-dimensional growth, so called anti-trees and partial antitrees \cite{Sa-AT, Sa-OC}.
The key point was a one-channel structure and the use of $2\times 2$ transfer matrices.

Inspired from the rather complicated form of the transfer matrices appearing in \cite{Sa-OC} we generalize this method further in a way that it can be in principle applied to any locally finite hopping operator, including models on  higher dimensional graphs like the Anderson model in $\ell^2(\ZZ^d)$. The main Theorem~\ref{Theo-1} is a generalization of a spectral averaging formula well known for Jacobi operators giving the spectral measure at the root as a limit of absolutely continuous measures involving an expression depending on transfer matrices. This formula can be found in the book by Carmona and Lacroix \cite[Theorem III.3.2 and III.3.6]{CL} and has been used for showing 
absolutely continuous spectrum for random one-dimensional models with fast enough decaying random potentials  \cite{KLS, LS}. A generalization of this formula to one-channel operators is also a key point in \cite{Sa-AT, Sa-OC}.
Unlike in these cases, here we have to deal with affine spaces of rectangular transfer matrices at each level and products of such spaces where the dimension grows. 
However, re-interpreting the product of transfer matrices as a transformation of a natural partial semi-group structure (associativity) of boundary resolvent data permits a generalization of this formula.
Furthermore, in this formula a maximum (or minimum in the denominator) for the choices of the transfer matrices within the affine spaces appears so that in principle one could try to find some specific choice to prove that a measure smaller than the actual spectral measure contains some absolutely continuous part. Then, the actual spectral measure will also contain some absolutely continuous part (cf. Theorem~\ref{Theo-1a}, Theorem~\ref{Theo-4} ).  

Whether this method can be used to solve the extended states conjecture remains to be seen as hard estimates on resolvents have to be done and a good choice of the transfer matrices has to be found. But we expect it to be 
usable for Anderson type models in between the antitree and the $\ZZ^d$ case to approach the open problem and state the difficulties more precisely.
As an immediate application we give a Last-Simon type proof \cite{LS} for absolutely continuous spectrum for random decaying shell-matrix potentials connecting one-dimensional wires, where the number of wires can be increasing and go towards infinity (cf. Theorem~\ref{Theo-5}).
This is already a non-one-dimensional problem and one has some (random) stair-like graph. As a by-product we also obtain the result for random decaying potentials on the strip as in \cite{FHS3} (cf. Theorem~\ref{Theo-7}). 
We also obtain absolutely continuous spectrum for random decaying shell-matrix potential on the Bethe lattice as a corollary (cf. Theorem~\ref{Theo-6}). 
This may seem weak as the Anderson model with iid potential (diagonal shell potential) as a.c. spectrum, however, the statement includes radially symmetric potentials where the whole model is equivalent to a sum of one-dimensional Jacobi operators and the decay is really needed.

Let us emphasize that the transfer matrix method for one and quasi-one dimensional models has been extremely fruitful in the past. For instance, by analyzing the transfer matrix cocycle the global theory for one-frequency quasi-periodic Schr\"odinger operators on $\ell^2(\ZZ)$  with analytic potential has been developed by Artur Avila \cite{Av} which was a major part for his Fields medal.
Many abstract general theorems relating the asymptotic behavior of transfer matrix products to spectral theory can be in some way linked to the spectral averaging formula mentioned above.
Therefore, we believe it can be a fruitful research program to revisit these relations and try to generalize such theorems using the spaces of transfer matrices as set up in this paper.

\subsection{Quasi-spherical partitions, radial structure}
Let $H$ be a symmetric operator on $\ell^2(\GG)$ with locally finite hopping as described above.
We may think of $\GG$ as a graph where $x,y\in\GG, x\neq y$ are connected if and only if $\langle \delta_x, H \delta_y\rangle\neq 0$.
Furthermore, let us assume that $\GG$ is connected with this graph structure, otherwise $H$ is an orthogonal sum of the operators on the connected components and we can reduce to one of them.
We will use the set up as in \cite{Sa-OC} and use a {\it quasi-spherical partition}, this means we partition $\GG$  into shells $S_n$ such that
$$
\GG=\bigsqcup_{n=0}^\infty S_n \qtx{with} \#(S_n):=s_n<\infty\;, 
$$ 
and
\begin{equation} \label{eq-cond-sph}
\langle \delta_x\,,\,H\,\delta_y\rangle\,=\,0 \qtx{if} x\in S_j\,,\,y\in  S_k \qtx{and} |j-k| \geq 2\;.
\end{equation}
We also define the sub-graphs
$$
\GG_{m,n}\,:=\,\bigsqcup_{j=m}^n S_j\qtx{for $m\leq n$}
$$
This means $H$ connects $S_n$ only to $S_n$ and $S_{n\pm 1}$. 
If $\GG$ with the graph structure from $H$ as described above is connected, then a possible way to obtain such a partition is to take the spherical partition where  $S_n=\{x\,:\,d(x,0)=n\}$ is the $n$-th sphere around some origin $0\in \GG$ and $d(x,y)$ denotes the graph distance\footnote{minimal number of edges passed when going from one point to another}.
For a standard Schr\"odinger operator in $\ZZ^d$ with the discrete Laplacian (graph adjacency operator) and some potential this leads to $S_n=\{x\in\ZZ^d\,:\,\|x\|_1=n\}$.

Choosing an order for all points in $S_n$ we identify $\ell^2(S_n)\cong \CC^{s_n}$ and using primarily the order of the shells we identify
$\ell^2(\GG_{m,n})$ with $\CC^{s_m+\ldots+s_n}$.
For $m\leq n$ we have canonical isometric embeddings 
$$
P_{m,n}\,:\, \CC^{s_m+\ldots +s_n}\cong \ell^2(\GG_{m,n}) \;\hookrightarrow\; \ell^2(\GG)\;,\quad 
P_n\,:=\,P_{n,n}\,:\,\ell^2(S_n)\;\hookrightarrow\; \ell^2(\GG)\,.
$$ 
Its adjoints $P_{m,n}^*$, $P_n^*$ are the natural projections from $\ell^2(\GG)$ onto $\ell^2(\GG_{m,n})$ o $\ell^2(S_n)$. 

We define the {\it partial operators} or {\it partial Hermitian matrices}  $H_{m,n}$ on  $\ell^2(\GG_{m,n})$, the $s_n\times s_n$ Hermitian matrix potentials $V_n$ and the $s_{n+1} \times s_n$ connection matrix $W_n$ by
\begin{equation}\label{eq-def-parts}
H_{m,n}\,:=\,P_{m,n}^* H P_{m,n}\;,\quad V_n\,:=\,P_n^* H P_n\;,\quad W_{n+1}\,:=\,P_{n+1}^* H P_n\;.
\end{equation}
Note that $V_n\in \Her(s_n)$ is a Hermitian $s_n \times s_n$ matrix.
The rank 
$$
r_n\,:=\,\rank\,W_n
$$
is at most the minimum of $s_{n-1}$ and $s_n$. In lattices $\ZZ^d$ using the spheres $S_n$ as mentioned above, we find that $s_n$ is non-decreasing and
$r_n$ is the maximal possible rank. 
However, one can use other decompositions and group together several spheres to one new shell in which case $r_n$ will not be maximal.

The polar or Hilbert-Schmidt decomposition of $W_n$ is given by $W_n=U_n D_n \hat U_{n-1}^*$
where $U_n\in \CC^{s_n\times r_n}$ and $\hat U_{n-1}\in \CC^{s_{n-1}\times r_n}$ are partial isometries,
$U_n^* U_n=\hat U_{n-1}^* \hat U_{n-1}=\II_{r_n}$ and $ D_n=\diag(D_{n,1},\ldots,D_{n,r_n})\,\in\,\RR^{r_n\times r_n}$ is diagonal and invertible with the non-zero singular values of $W_n$ along the diagonal. Here, we use the notation $\KK^{m\times n}$ for a $m\times n$ matrix with values in $\KK$ where $\KK=\RR$ or $\KK=\CC$.
By putting the singular value part $D_n$ to the left or right (or possibly splitting it partly to $U_n$ and partly to $U_{n-1}^*$) and introducing some (conventional) minus sign
we may write
\begin{align}
 W_n\,=\, -\Upsilon_n\,\Phi_{n-1}^* \qtx{where}
\Upsilon_n\,\in\,\CC^{s_n\times r_n}\;,\Phi_{n-1}\,\in\,\CC^{s_{n-1} \times r_n}\;
\end{align}
are all matrices of full rank. For example, we could choose $\Phi_n=\hat U_n$ and $\Upsilon_n=U_n D_n$.
However, the transfer matrix technique and formulas we develop will in fact 
work for any decomposition of this kind.  

Note that $W_0$ and $\Upsilon_0$ are not defined in this process (they do not appear as we do not have a $-1$st shell). We may select some root point $0\in S_0=\{0\}$ and let
$\Upsilon_0=P_0^*\delta_{0}\,\in\,\ell^2(S_0)\,\equiv\, \CC^{s_0}$ be the normalized vector supported on $\{0\}$. Or, we may also select some other non-zero {\it root-vector} $\Upsilon_0\in \ell^2(S_0)$.
In any case, we select $r_0=1$ so that $\Upsilon_0$ is a $s_0\times 1$ matrix of full rank 
which can be seen as a vector. 
In analogy of the notion of channels for strip-operators and the one-channel operators  we may refer to $\Upsilon_n$ as the backward channels at the $n$-th shell $S_n$ connected to  $S_{n-1}$
and $\Phi_n$ as the forward channels at $S_n$ connected to $S_{n+1}$. 

Using the direct Hilbert sum structure given by the shells
$$
\ell^2(\GG)=\bigoplus_{n=0}^\infty \ell^2(S_n) \,\cong\, \bigoplus_{n=0}^\infty \CC^{s_n}\;.
$$
we write $\psi\in\ell^2(\GG)$ as direct sum
\begin{equation}
\psi\,=\,\bigoplus_{n=1}^\infty\psi_n \qtx{where} \psi\,=\,P_n^* \psi\,\in\,\CC^{s_n}\cong \ell^2(S_n)\;.
\end{equation}
Then, by \eqref{eq-cond-sph} and \eqref{eq-def-parts} we find with the convention $\Phi_{-1}^*\psi_{-1}=0$ that
\begin{equation}\label{eq-H}
(H\psi)_n\,=\,-\,\Phi_n \Upsilon_{n+1}^* \,\psi_{n+1}\,-\,\Upsilon_{n}\,\Phi_{n-1}^*\psi_{n-1}\,+\,V_n\,\psi_n\;.
\end{equation}

\subsection{Boundary resolvent data and transfer matrices}

We define the $m$ to $n$ boundary data, $m\leq n$, at the spectral parameter $z \not \in \spec(H_{m,n})$ by
\begin{equation}
R^z_{m,n}\,:=\,\pmat{ (P_{m,n}^* P_m \Upsilon_m)^* \\ (P_{m,n}^* P_n \Phi_n)^*} (H_{m,n}-z)^{-1}\,\pmat{P_{m,n}^* P_m \Upsilon_m & P_{m,n}^* P_n \Phi_n}
\end{equation} 
Note that $\Upsilon_m$ (or its column vectors) represent the modes connecting $S_m$ and therefore also $\GG_{m,n}$ to $S_{m-1}$ 
where as $\Phi_n$ represents the modes connecting $\GG_{m,n}$ to the next shell $S_{n+1}$.
$P_{m,n}^* P_m$ just represents the natural embedding of $\ell^2(S_m)$ into $\ell^2(\GG_{m,n})$ meaning that we consider the column-vectors of $\Upsilon_m$ and $\Phi_n$ now as vectors in $\ell^2(\GG_{m,n})$ for obtaining the resolvent data $R^z_{m,n}$.
This means, $\Upsilon_m$ is a $s_m\times r_m$ matrix and $P_{m,n}^* P_m \Upsilon_m=\smat{\Upsilon_m \\ \nul}$ is a $\#(\GG_{m,n}) \times r_m$ matrix, similarly $P_{m,n}^* P_n \Phi_n=\smat{\nul \\ \Phi_n}$ is a
$\#(\GG_{m,n}) \times r_{n+1}$ matrix. 
Hence, $R^z_{m,n}$ is a $(r_m+r_{n+1}) \times (r_m+r_{n+1})$ matrix and we 
assign to it its natural $(r_m, r_{n+1})$ splitting and define
$$
\alpha_{m,n}^z\,\in\,\CC^{r_m \times r_m}\;,\quad \beta^z_{m,n}\,\in\,\CC^{r_m \times r_{n+1}}\;,\quad
\gamma^z_{m,n}\,\in\,\CC^{r_{n+1} \times r_m}\;,\quad \delta^z_{m,n}\\,\in\,\CC^{r_{n+1}\times r_{n+1}}\;.
$$
by
\begin{equation}
R^z_{m,n}\,=:\,\pmat{\alpha^z_{m,n} & \beta^z_{m,n} \\ \gamma^z_{m,n} & \delta^z_{m,n}}\;.
\end{equation}
This means, e.g., that $\alpha^z_{m,n}=\Upsilon^*_m P_m^* P_{m,n} (H_{m,n}-z)^{-1} P_{m,n}^* P_m \Upsilon_m$. 
Also note the general relations
\begin{equation}
\alpha^{\bar z}_{m,n}\,=\,(\alpha^z_{m,n})^*\;, \quad
\gamma^{\bar z}_{m,n} = (\beta^z_{m,n})^*\;,\quad \delta^{\bar z}_{m,n}=(\delta^z_{m,n})^*\;.
\end{equation}
For $m=n$ we have the $n$-th shell boundary data
\begin{equation}
R^z_n\,=\,\pmat{\alpha^z_n & \beta^z_n \\ \gamma^z_n & \delta^z_n}\,:=\,\pmat{\Upsilon_n^* \\ \Phi_n^*}\,(V_n-z)^{-1}\,\pmat{\Upsilon_n & \Phi_n}\,.
\end{equation}
For some $z\in\spec(H_{m,n})$ one may define $R^z_{m,n}$ by analytic extension of $z\to R^z_{m,n}$.
Therefore, we define the first set of singular spectral parameters by
\begin{equation}
A^1_{m,n}\,:=\,\{\lambda\in \RR\,:\,R^\lambda_{m,n}\;\text{is not defined ater analytic extension}\;\}\;,\quad A^1_n\,:=\,A^1_{n,n}\;.
\end{equation}
So for $z\not \in A^1_{m,n}$ the matrix $R^z_{m,n}$ exists. Clearly, $A^1_{m,n}$ is a subset of the eigenvalues of $H_{m,n}$ and therefore finite.
We will make some additional assumptions:

\vspace{.2cm}

\noindent {\bf Assumptions:}\\
\noindent {\bf (A1):} $(r_n)_n$ is non-decreasing,
$
r_{n+1}\,\geq\;r_n\;.
$
\\
{\bf (A2):} We assume that for all $n\in\NN_0$ and almost all real $z=\lambda \in \RR$
 we have that the $r_n \times r_{n+1}$ matrix $\beta^\lambda_n$ has full rank $r_n$.

\vspace{.2cm}  
  
The condition (A2) in principle says that the matrix $V_n$ connects all modes of $\Upsilon_n$ and $\Phi_n$. This will be the case for typical operators like Schrödinger operators (adjacency+potential) on a connected graph.
In fact, if $\beta^z_n$ has full rank for some real $z=\lambda$, then by considering the rational function $z\mapsto \det( \beta^z_n (\beta^{\bar z}_n)^*)$ we see that it has full rank for all but finitely many $z\in \CC$.
The condition (A1) can always be obtained after grouping shells together.
If the $\liminf$ of $r_n$ is bounded one can group shells together such that the rank is (eventually) constant\footnote{one has a sub-sequence $n_k$ with 
$r_{n_k}=r=\liminf r_n$ and one can define new shells $S'_k:=\GG_{n_k,n_{k+1}-1}$ (defining $n_0=0$) to get a partition of $\GG$ with property \eqref{eq-cond-sph} and ranks $r'_k=r$.}. This is particularly the case for strip-like structures, but also for anti-trees. Otherwise, one can group the shells such that $r_n$ is increasing. Now, if one groups together $S_m$ up to $S_n$ (with $r_{n+1}\geq r_m$), then one needs
$r_i \geq r_m$ for any $m\leq i \leq n$ in order to have a chance of fulfilling (A2). But the grouping can always be done in this way.

Under assumptions (A1) and (A2) we define
$$
 A^2_{m,n}\,:=\,\{z\,\in\,\CC\,:\,\rank\,\beta^z_{m,n} < r_m\,\}\;,\; A_{m,n}\,:=\,A^1_{m,n}\cup A^2_{m,n}\;,\; A_n\,:=\,A_{n,n}\;.
$$  
With the assumptions above we immediately get that all $A_n$ are finite sets. We will show that any set $A_{m,n}$ is also finite.
We call $A_n$ the set of singular parameters at the shell $n$.
In the work on one-channel operators \cite{Sa-OC} we had $r_n=1$ for all $n$ and $\alpha^z_n$, $\beta^z_n$ and so on where numbers. In this case the transfer matrices featured some inverse of $\beta^z_n$.
As we have rectangular shaped matrices here, we need to make some different sense of this inverse.
If $\beta^z_n$ is of full rank, then because $r_n\leq r_{n+1}$ there exists an affine space of right-inverses which we denote by $\BB^z_n$, similarly we define $\BB^z_{m,n}$.
\begin{equation*}
\BB^z_{m,n}\,:=\,\{B\,\in\,\CC^{r_{n+1}\times r_m}\,:\, \beta^z_{m,n}\,B\,=\,\II_{r_m}\,\}\;,\quad
\BB^z_n\,:=\,\BB^z_{n,n}\,\subset\,\CC^{r_{n+1}\times r_n }\;,
\end{equation*}
where $\II_r=\diag(1,\ldots,1)$ denotes the $r\times r$ unit matrix. 
We furthermore define
\begin{equation*}
\BBb^z_{m,n}:=\,\{\bbb\,\in\,\CC^{r_{n+1}\times r_m}\,:\, \beta^z_{m,n}\,\bbb\,=\,\alpha^z_{m,n}\;\}\;,\quad
\BBb^z_n\,:=\,\BBb^z_{n,n}\,\subset\,\CC^{r_{n+1}\times r_n }\;.
\end{equation*}
If $\alpha^z_{m,n}$ is invertible, then one simply has $\BBb^z_{m,n} = \BB^z_{m,n} \alpha^z_{m,n}$ as sets, but if $\alpha^z_{m,n}$ is not invertible (extreme case $\alpha^z_{m,n}=\nul$) then  $\BBb^z_{m,n}$ may be a bigger set.
In fact, if $\beta^z_{m,n} B=\II$ and $\KK^z_{m,n}=\ker(\beta^z_{m,n})^{\otimes r_m}$ is the space of $r_{n+1} \times r_m $ matrices where each column vector is in the kernel of $\beta^z_{m,n}$, then
$\BB^z_{m,n}=B+\KK^z_{m,n}$ and $\BBb^z_{m,n}=B\alpha^z_{m,n} + \KK^z_{m,n}$.
All these sets are well defined for $z\not\in A_{m,n}$ and $\BB^z_{n},\; \BBb^z_{n}$ are defined for $z\not\in A_{n}$.
In analogy to the transfer matrix formulas for one-channel operators in \cite{Sa-OC} we define the sets of transfer matrices for $z\not \in A_{m,n}$ by
\begin{equation}
\TT^z_{m,n}\,:=\,\left\{\pmat{B & - \bbb   \\ \delta^z_{m,n} B  & \gamma^z_{m,n}-\delta^z_{m,n} \bbb }\,:\, B\,\in\,\BB^z_{m,n}\,,\;\bbb\,\in\,\BBb^z_{m,n} \right\}\;,
\;\;\TT^z_n\,:=\,\TT^z_{n,n}\;. \label{eq-tr-1}
\end{equation}
Note that $\TT^z_{m,n}$ is an affine subspace of $\CC^{2r_{n+1} \times 2 r_m}$, the set of $2r_{n+1} \times 2r_m$ matrices.
We also define some simpler affine subspaces for $z\not\in A_n$ by
\begin{equation}
\T^z_{m,n}\,:=\,\left\{\pmat{B & - B\alpha^z_{m,n} \\ \delta^z_{m,n} B & \gamma^z_{m,n}-\delta^z_{m,n} B \alpha^z_{m,n}}\,:\, B\,\in\,\BB^z_{m,n} \right\}\,\subset\,
\TT^z_{m,n}\;,\;\T^z_n\,:=\,\T^z_{n,n}\; \label{eq-tr-2}
\end{equation}
Furthermore, for $z\not\in A_{0,n}$  we define the set of boundary vectors at $n$ with initial Dirichlet condition
\begin{equation}
\D^z_{n}\,:=\,\left\{\pmat{B \\ \delta^z_{m,n} B}\,:\,B\,\in\,\BB^z_{0,n} \right\}\,=\, \TT^z_{0,n} \pmat{1 \\ 0} \,=\,\T^z_{0,n}\pmat{1\\0} \subset\,\CC^{2r_{n+1} \times 1}\,\equiv\,\CC^{2r_{n+1}}\;
\end{equation}
as well as the set of vectors with initial von Neumann condition
\begin{equation}
\N^z_{n}\,:=\,\left\{\pmat{-\bbb \\ \gamma^z_n-\delta^z_{m,n} \bbb}\,:\,\bbb\,\in\,\BBb^z_{0,n} \right\}\,=\, 
\TT^z_{0,n} \pmat{0 \\ 1} \,\subset\,\CC^{2r_{n+1} \times 1}\,\equiv\,\CC^{2r_{n+1}}\;.
\end{equation}
Recall that $r_0=1$ and therefore $\TT^z_{0,n},\,\T^z_{0,n}$ are sets of $2r_{n+1} \times 2$ matrices.

\begin{rem} \label{rem-T}
{\rm (i)} Let us quickly note, that these matrices are a generalization of the transfer matrices for block-Jacobi operators.
For instance, take $S_n=\{n\}\times S$ on $\GG=\ZZ_+\times S$ with $\#(S)=s$ finite, $\psi_n\in \ell^2(S)\equiv \CC^s$ and an operator of the form
$$
(H\psi)_n\,=\,-\psi_{n-1}-\psi_{n+1}+V_n \psi_n\,,\quad n\geq 0,\quad \psi_{-1}=0\,.
$$
Then, for $n\geq 1$ we can select $\Upsilon_n=\Phi_n=\II_s$ and $\Phi_0=\II_s$ as well.
This leads to $\alpha^z_n=\beta^z_n=\gamma^z_n=\delta^z_n= (V_n-z\II_s)^{-1}$ for $n\geq 1$ and there is only one choice, $B= (\beta^z_n)^{-1}$ and 
$\bbb=(\beta^z_n)^{-1} \alpha^z_n$ leading to
$$
\TT^z_n=\{T^z_n\}\,,\quad T^z_n=\pmat{(\beta^z_n)^{-1} & -(\beta^z_n)^{-1} \alpha^z_n \\ \delta^z_n (\beta^z_n)^{-1}& \gamma^z_n-\delta^z_n(\beta^z_n)^{-1} \alpha^z_n}\,=\,\pmat{V_n-z\II_s & -\II_s \\ \II_s & \nul}
$$
which is the standard $n$-th transfer matrix for such a block Jacobi operator.
Only for $n=0$ we would have some different set of $2s \times 2$ transfer matrices, depending on the selection of the vector $\Upsilon_0\in\CC^s \equiv \ell^2(S_0)$.\\
{\rm (ii)} Let us also note that the transfer matrices appearing as products $A_i B_i$ in \cite{FHS} are elements of these defined sets $\TT^z_i$ of transfer matrices. More precisely, there  we have the special choices $\Phi_n=\II$ and $\Upsilon_n=-W_n$ for all $n$, which are supposed to be full rank so that
$\Lambda_n\,:=\,\Upsilon_n (\Upsilon_n^* \Upsilon_n)^{-1}$ exists, giving $\Upsilon_n^* \Lambda_n=\II$. Then, choose $B=(V_n-z\II)\Lambda_n$ and $\bbb=\Upsilon_n$ giving
the special transfer matrix $T^z_n\,=\,\smat{(V_n-z\II)\Lambda_n & - \Upsilon_n \\ \Lambda_n & \nul}$ and we find 
$$ \TT^z_n=\left\{ \pmat{(V_n-z\II)(\Lambda_n+K_1) & -\Upsilon_n +(V_n-z\II)K_2 \\ \Lambda_n+K_1 & K_2}\,:\,\Upsilon_n^*K_i=\nul\right\}\,.$$
\end{rem}

For subsets of matrices $\TT_1, \TT_2$ we denote $\TT_2 \TT_1:=\{T_2T_1\,:\,T_i\in\TT_i\}$ if the matrices can be multiplied this way.
We now define some further sets of singular spectral parameters.
For $l\leq m < n$ let
\begin{equation}
B_{l,m,n}\,:=\,\left\{\lambda \in \RR\,:\, \II_{r_{m+1}}\,-\,\alpha^\lambda_{m+1,n} \delta^\lambda_{l,m}\;\;\text{is not invertible}\;\right\}
\end{equation}
which is a finite set.
This can be easily seen as the determinant is a non-zero rational function as  for $\im(z)>0$ the inverse
$$
\left(\II_{r_{m+1}}\,-\,\alpha^z_{m+1,n} \delta^z_{l,m}\right)^{-1}\,=\,
(\delta^z_{l,m})^{-1} \left((\delta^z_{l,m})^{-1}-\alpha^z_{m+1,n}\right)
$$
exists, using $\Im(\delta^z_{l,m})^{-1}<0$ and $\Im(\alpha^z_{m+1,n})>0$ where 
$\Im(M)=(M-M^*)/(2i)$ is the imaginary part in the $C^*$ algebra sense.

\begin{prop}\label{prop-1}
Under assumptions {\rm (A1)} and {\rm (A2)} we have the following:\\
{\rm (i)} For all $l,n\in\NN_0$, $l\leq n$ the set $A_{l,n}$ is finite.\\
{\rm (ii)} For $0\leq l \leq m \leq n-1$, $l,m,n\in\NN_0$ and
$z\not \in A_{l,m} \cup A_{m+1,n} \cup B_{l,m,n}$ we find $z\not \in A_{l,n}$ and
$$
\TT^z_{m+1,n}\,\TT^z_{l,m}\,\subset\, \TT^z_{l,n}\;,\quad
\TT^z_{m+1,n} \TT_{l,m} \pmat{\nul \\ \II_{r_l}}\,=\,\TT^z_{l,n} \pmat{\nul \\ \II_{r_l}}\,.
$$
as well as
$$
\TT^z_{m+1,n} \TT_{l,m} \pmat{\II_{r_l} \\ \nul}\,=\,\T^z_{m+1,n} \T^z_{l,m} 
\pmat{\II_{r_l} \\ \nul} \,=\, \TT^z_{l,n} \pmat{\II_{r_l} \\ \nul}\,.
$$
{\rm (iii)}
Particularly, we find for $z\not \in A_{0,m} \cup A_{m+1,n} \cup B_{0,m,n}$, that
$$
\TT^z_{m+1,n} \D^z_{m}\,=\,\T^z_{m+1,n} \D^z_m \,=\,\D^z_n\qtx{and}
\TT^z_{m+1,n} \N^z_m\,=\,\N^z_n
$$
{\rm (iv)} For fixed $n\in \NN$ and all but finitely many $z\in \CC$ we have
$$
\D^z_n\,=\,\TT^z_n \TT^z_{n-1}\cdots \TT^z_1 \TT^z_0 \pmat{1\\0}\,=\, \T^z_n \T^z_{n-1}\cdots \T^z_1 \T^z_0 \pmat{1\\0}\;.
$$
\end{prop}

Before coming to the spectral theory, let us mention that we have some symplectic structure as in the Jacobi case.

\begin{defini}\label{def-HSP}
First, by $J_m$ we denote the standard $2m \times 2m$ symplectic matrix, this means
$$
J_m\,:=\,\pmat{\nul & -\II_m \\ \II_m & \nul}
$$
where $\II_m$ denotes the unit $m \times m$ matrix, as before.
We define the Hermitian-symplectic partial semi-group $\HSP$ and the symplectic partial semi-group
$\SP$ by
$$
\HSP:=\bigcup_{\substack{m\geq n \\ m,n\in\NN}}\{ T\in \CC^{2m \times 2n}\,:\, T^* J_m T=J_n\}\,,\;\;
\SP:=\bigcup_{\substack{m\geq n \\ m,n\in\NN}}\{ T\in \CC^{2m \times 2n}\,:\, T^\top J_m T=J_n\}
$$
\end{defini}

Note,  if $T_1, T_2 \in \HSP$ or $\SP$ are such that one can multiply $T_1$ with $T_2$ from the left, then $T_2 T_1\in \HSP$ or $\SP$ as well.
Note,  ${\rm HSP}(2m)=\HSP \cap \CC^{2m\times 2m}$ is the Hermitian-symplectic group, ${\rm SP}(2m,\CC):=\SP\cap \CC^{2m\times 2m}$ is the complex symplectic group and
${\rm Sp}(2m,\RR):=\HSP\cap\SP\cap \CC^{2m\times 2m}$ is the real symplectic group of $2m\times 2m$ matrices.

For Jacobi and block-Jacobi operators it is well known that we have a symplectic structure for the transfer matrices. Here, there is something similar:

\begin{prop}\label{prop-2}
{\rm (i)} For $z=\lambda \in \RR \setminus A_{m,n}$ we have $\TT^\lambda_{m,n}\subset \HSP$. Moreover, for any $T_1, T_2 \in \TT^\lambda_{m,n}$ we find
$T_1^* J_{r_{n+1}} T_2\,=\, J_{r_m}$.
\\
{\rm (ii)} If all $V_n$ and all $\Upsilon_n$ and $\Phi_n$ for $n\in\NN_0$ are real matrices, then we have for $z\in\CC \setminus A_{m,n}$ that $ \TT^z_{m,n} \subset \SP$. Moreover, for any $T_1, T_2 \in \TT^\lambda_{m,n}$ we find $T_1^\top J_{r_{n+1}} T_2= J_{r_m}$.
\end{prop}

\begin{rem}
If we have a real operator $H$, then $V_n$ are real and we can choose all $\Upsilon_n,\,\Phi_n$ to be real valued. In this case, for real energies $z=\lambda\in\RR$ the transfer matrix sets are subsets of $\HSP\cap\SP$. Now, as mentioned above, the square matrices in this set are part of a real symplectic group and particularly real. In $\HSP\cap \SP$ the non-square matrices are not necessarily real. However, in this particular case, it would be sufficient to restrict to the set of real transfer matrices.
\end{rem}

Like in the Jacobi case and the one-channel case there is also a connection to solutions of the formal eigenvalue equation. However, we need to exclude a further set of spectral parameters which is when $\bar z \in A_n$ for some $n$. So on the real line we do not have any additional exclusions.

\begin{prop}\label{prop-3}
{\rm (i)}
Let $z, \bar z \not \in \bigcup_{n=0}^\infty A_n$  and let $\Psi=(\Psi_n)_{n=0}^\infty \in \prod_{n=0}^\infty \CC^{S_n} \,=\,\CC^\GG$ be a formal solution of 
$H\Psi=z\Psi+v P_0 \Upsilon_0$, i.e.
$(H\Psi)_n=z \Psi_n$ for $n\geq 1$ and $(H\psi)_0=z\Psi_0+ v \Upsilon_0$ with $v \in \CC$. 
Assume that $v\neq 0$ or $u:=\Upsilon_0^* \Psi_0 \neq 0$.
Then, there exist $T^z_n \in \TT^z_n$, $n\in \NN_0$, such that
$$
\pmat{\Upsilon_{1}^* \Psi_{1} \\ \Phi_0^* \Psi_0 }\,=\, T^z_0 \pmat{u \\ v }\qtx{and}
\pmat{\Upsilon_{n+1}^* \Psi_{n+1} \\ \Phi_n^* \Psi_n}\,=\,T^z_n \,\pmat{\Upsilon_{n}^* \Psi_{n} \\ \Phi_{n-1}^* \Psi_{n-1}}\,.
$$
In particular, we find
$$
\pmat{\Upsilon_{n+1}^* \Psi_{n+1} \\ \Phi_n^* \Psi_n }\,=\, T^z_n T^z_{n-1} \cdots T^z_0\, \pmat{\Upsilon_0^* \Psi_0 \\ v }\,.
$$
{\rm (ii)} Let  $z\not \in \bigcup_{n=0}^\infty A_n $, let be given a vector $\smat{u\\v}\in \CC^2$ and a selection of transfer matrices $T^z_n \in \TT^z_n$. Then, there exists a formal solution $\Psi=(\Psi_n)_{n=0}^\infty \in \prod_{n=0}^\infty \CC^{S_n}=\CC^\GG$ of
$H\Psi=z\Psi+v P_0 \Upsilon_0$ such that
$$
\Upsilon_0^* \Psi_0\,=\,u \qtx{and} \pmat{\Upsilon_{n+1}^* \Psi_{n+1} \\ \Phi_n^* \Psi_n }\,=\, T^z_n T^z_{n-1} \cdots T^z_0\, \pmat{u \\ v }\,.
$$
\end{prop}

Note, the correspondences in $a)$ and $b)$ are not one-to-one in general!
If we let $v=w\cdot u=w\cdot \Upsilon_0^* \Psi_0$, then $\Psi$ is a formal solution of $(H-w P_0\Ups_0 \Ups_0^* P_0^*) \Psi = z \Psi$
and one can relate to $v=w\cdot u$ as some form of 'boundary condition'.

\section{Main Results}

\subsection{Spectral Theory}
So far, everything can be well defined for unbounded symmetric operators with local finite hopping.
For spectral theory we need a self-adjoint operator. 
As mentioned above, we want $H$ to be self-adjoint with its maximal domain. 

\begin{defini}
A Hermitian (possibly unbounded) operator $H$ on $\ell^2(\GG)$ with locally finite hopping will be said to be {\bf self-adjoint with its natural domain}, 
if the set of compactly supported vectors, $\HH_{c}:=\{\psi\in\ell^2(\GG)\,:\, \#\{x\in\GG\,:\,\psi(x)\neq 0\}\,<\,\infty\}$, is a core for $H$, i.e. $H_{min}=H|\HH_{c}$ is essentially self-adjoint. 
\end{defini}

$H_{min}^*$ has the maximal possible domain in terms of restricting $H$ (as operator on $\CC^\GG$) and its range to $\ell^2(\GG)$, 
i.e. the domain of $H_{min}^*$ is given by all functions $\psi$ in $\ell^2(\GG)$ such that 
$H\psi$ is as well in $\ell^2(\GG)$. This is the natural domain for $H$ in $\ell^2(\GG)$. 
Note that $H_{min}$ is essentially self-adjoint if and only if $H_{min}^*$ is self-adjoint in which case $\HH_c$ is a core.
In the language of physicists this property means that one does not have to specify a 'boundary condition at infinity'.
It is evident that $\ell^2$ operator-norm bounded operators $H$ have this property.
From now on by $H$ we consider the operator with its natural domain, i.e. 
$H=H_{min}^*$.

If $H=H_{min}^*$ is self-adjoint then we will denote the spectral measure at the chosen root-vector $P_0 \Upsilon_0\in \ell^2(\GG)$ by $\mu$, i.e. 
$$
\int f(\lambda)\,\mu(d\lambda)\,=\, \langle\,P_0 \Upsilon_0\,,\, f(H)\,P_0 \Upsilon_0 \,\rangle\;.
$$

\begin{Theo}\label{Theo-1} Assume that $H$ is self-adjoint with its natural domain and that assumptions (A1), (A2) hold.
There is a point measure $\nu$ supported on $\liminf_n A^1_{0,n}=\bigcup_{m\geq 0} \bigcap_{n\geq m} A^1_{0,n}$ and generated by
compactly supported eigenfunctions of $H$, such that
\begin{align*}
d\mu(\lambda)\,&=\,d\nu(\lambda)\,+\, \lim_{n\to\infty}\,\frac{1}{\pi} \left\| \left(1+(\delta^\lambda_{0,n})^2\right)^{-1/2} \gamma^\lambda_{0,n}\,\right\|^2 \,d\lambda \\
&\,=\,d\nu(\lambda)\,+\,\lim_{n\to\infty}\,\frac{1}{\pi}\,\left( \min_{\bfu^\lambda_n \in\D^\lambda_n}\|\bfu^\lambda_n\|^2  \right)^{-1}\,d\lambda \\
&=\, d\nu(\lambda)\,+\,\lim_{n\to\infty} \frac1\pi \max_{\substack{T^\lambda_i\in \TT^\lambda_i \\ i=1,\ldots,n}}\,\left\| T^\lambda_n T^\lambda_{n-1}\cdots T^\lambda_{1} T^\lambda_0 \pmat{1\\0} \right\|^{-2}\;d\lambda\\
&=\, d\nu(\lambda)\,+\,\lim_{n\to\infty} \frac1\pi \max_{\substack{T^\lambda_i\in \T^\lambda_i \\ i=1,\ldots,n}}\,\left\| T^\lambda_n T^\lambda_{n-1}\cdots T^\lambda_{1} T^\lambda_0 \pmat{1\\0} \right\|^{-2}\;d\lambda
\end{align*}
where $d\lambda$ denotes the Lebesgue measure in $\RR$ and the limit has to be understood in the weak topology of finite Radon-measures (i.e. $d\mu_n(\lambda)\to d\mu(\lambda)$ iff $\int f(\lambda) d\mu_n(\lambda) \to \int f(\lambda)d\mu(\lambda)$ for bounded continuous functions $f$).
\end{Theo}

Note that this is a generalization of  \cite[Theorem~III.3.2]{CL}, \cite[Theorem~2.3(iii)]{Sa-AT} and \cite[Theorem~2(i)]{Sa-OC}.
As in the cases before, the proof uses some spectral averaging technique in the $n$-th shell for the operator $H_{0,n}$ (restriction to the of $H$ to $\GG_{0,n}$) which selects a certain point within the Weyl discs.
Then, the limit point property in the case of self-adjointness reflects the fact that the averaged spectral measures converge to the actual spectral measure of $H$ for $n\to\infty$.
Here, we have the following Weyl discs:

\begin{Theo}\label{Theo-2}
Let $\im(z)>0$ let the $n$-th Weyl disc be given by
$$
\WW^z_n\,=\, {\rm cl\,}\{\alpha^z_{0,n}+\beta^z_{0,n} A(\II-\delta^z_{0,n}A)^{-1} \gamma^z_{0,n}\,:\, A\in \CC^{r_{n+1} \times r_{n+1}}\,,\, \Im(A)\,\geq\,\nul\;\}\,\subset\,\CC\;,
$$
where $\Im(A):=(A-A^*)\,/\,(2\imath)$ denotes the 'imaginary part' of $A$ in $C^*$ algebra sense and ${\rm cl\,} \Ss$ denotes the topological closure of a set $\Ss\subset \CC$.
Assume that assumptions (A1) and (A2) hold, then:
\begin{enumerate}[{\rm (i)}]
\item $\WW^z_n$ is a closed (circular) disc in the upper half plane, $\WW^z_{n+1} \subset \WW^z_n$  and $\bigcap_{n=0}^\infty \WW^z_n$ is either a closed limit disc with positive radius ({\bf limit circle case}) or it consists of only one point ({\bf limit point case}).
\item If one of the deficiency indices $d^\pm=\dim \ker (H_{min}^* \pm \imath)$ is zero (particularly if $H_{min}^*$ is self-adjoint) then $\bigcap_{n=0}^\infty \WW^z_n$ consists of one point. If $H=H_{min}^*$ is self-adjoint, then
$$
\left\{\,\langle P_0 \Upsilon_0\,,\,(H-z)^{-1}\, P_0 \Upsilon_0\,\rangle\,\right\}\,=\, \bigcap_{n=0}^\infty \WW^z_n\,. 
$$
\item Let $z,\bar z\not \in \bigcup_{n=0}^\infty A_n$. Then, $\bigcap_{n=0}^\infty \WW^z_n$ is a limit disc of positive radius if and only if there are $\ell^2$ functions $\psi^\pm \in \ell^2(\GG)$ with $\langle P_0\Upsilon_0, \psi^\pm\rangle \neq 0$ and $H_{min}^*\psi^+=z\psi^+$ and $H_{min}^*\psi^-=\bar z \psi^-$.
\end{enumerate}
\end{Theo}

\begin{rem}
{\rm (i)}
For semi-infinite Jacobi operators (semi-infinite tri-diagonal Hermitian matrices) it is well known that the limit point property is equivalent to $H_{min}^*$ being self-adjoint. 
Furthermore, in the case where $\bigcap_{n} \WW^z_n$ is a limit disc, each point on the boundary (the limit circle) characterizes a self-adjoint extension of $H_{min}$.
A nice proof of these facts from the discrete analogue of Weyl-Titchmarch theory can be found in Techl's monograph on Jacobi operators \cite{Tes}.
For one-channel operators we have shown the equivalence under certain conditions, for instance, if all matrix entries of $H$ are real \cite{Sa-OC}. \\
{\rm (ii)} 
One might work with some matrix $\Upsilon_0$ (using $r_0>1$) consisting of column vectors in $\ell^2(S_0)$ and Weyl-discs (with same formula as above) of higher dimensions where one can have
different types of limit discs (different dimensions). In the case of  Jacobi operators with matrix entries (operators on strips) this is a typical choice to get the usual $2m\times 2m$ transfer matrices. The Weyl-Titchmarsh theory in this case is developed in \cite{Ber} and for the nice geometric picture of these higher dimensional Weyl discs see \cite{Fuk, SB}.
\end{rem}

Theorem~\ref{Theo-1} gives the following useful criteria.
By $\spec(H)$, $\spec_{ac}(H)$ we denote the spectrum and the absolutely continuous spectrum of $H$ as closed sets.

\begin{Theo}\label{Theo-1a}
{\rm (i)}
Let be given an interval $[a,b]\subset \RR$. Suppose for some increasing sequence $(n_k)_k$ of positive integer and all $\lambda\in[a,b]$ there is a measurable function
$\lambda\mapsto T^\lambda_{0,n_k} \in \TT^\lambda_{0,n_k}$, such that 
$$
\sup_{k\in\NN} \left\|T^\lambda_{0,n_k} \pmat{1\\0}\right\|\,<\,\infty\quad \text{for Lebesgue almost all $\lambda\in(a,b)$}\;.
$$
Then, $[a,b]\subset \supp \mu_{ac} \subset \spec_{ac}(H)$ meaning that the measure $\mu$ has an absolutely continuous part whose support contains $[a,b]$. In other words, there is absolutely continuous spectrum everywhere in $[a,b]$.\\[.2cm]
{\rm (ii)}
Suppose that we find some almost everywhere positive, locally integrable function $w(\lambda)\in L^1_{loc}1{(a,b)}$ and measurable functions $\lambda \mapsto T^\lambda_{0,n}\in \TT^\lambda_{0,n}$ such that
for any compact set $K\subset (a,b)$ of positive Lebesgue measure we have
$$
\liminf_{n\to \infty} \int_K \log\left( \left\|T^\lambda_{0,n} \pmat{1\\0} \right\|\,w(\lambda)\,\right)\,w(\lambda)\,d\lambda\,<\,\infty\;.
$$
Then, $[a,b] \subset \supp_{ac} \mu \subset \spec_{ac}(H)$, i.e. there is absolutely continuous spectrum everywhere in $[a,b]$.
\end{Theo}

Part (ii) is using the criterion of Deift-Killip \cite{DK} in their work on absolutely continuous spectrum for Schrödinger operators with $\ell^2$ potentials on $\ell^2(\ZZ)$.
Now, with some stronger condition we get to pureness of absolutely continuous spectrum like in the criterion of Last-Simon \cite{LS}.

\begin{Theo}\label{Theo-4}
\begin{enumerate}[\rm (i)]
\item 
Suppose for some $p>1$, $[a,b]\subset \RR$  and $\lambda\in [a,b]$ one finds measurable functions $\lambda\mapsto T^\lambda_{0,n} \in \TT^\lambda_{0,n}$ and $\lambda \mapsto u_n(\lambda)\in \CC$ such that
$$
\liminf_{n\to\infty} \int_a^b \left\| T^\lambda_{0,n} \pmat{u_n(\lambda)\\1} \right\|^{2p}\,d\lambda\,<\,\infty\;.
$$
Then, the positive measure $\mu-\nu$ (as in Theorem~\ref{Theo-1}) restricted to $(a,b)$ is purely absolutely continuous with a density which is in $L^p(a,b)$.
 Hence, $\mu$ is purely absolutely continuous in $(a,b) \setminus \liminf_n A^1_{0,n}$.
\item 
Suppose for some $p>1$, $[a,b]\subset \RR$  and $\lambda\in [a,b]$ one finds a measurable function $\lambda\mapsto T^\lambda_{0,n} \in \TT^\lambda_{0,n}$ such that
$$
\liminf_{n\to\infty} \int_a^b \| T^\lambda_{0,n}  \|^{2p}\,d\lambda\,<\,\infty
$$
 then $[a,b] \subset \supp \mu_{ac} \subset \spec_{ac}(H)$ and the measure $\mu-\nu$ is purely absolutely continuous in $(a,b)$ with an $L^p$ density.
 Hence, $\mu$ is purely absolutely continuous in $(a,b) \setminus \liminf_n A^1_{0,n}$ and there is spectrum.
\end{enumerate}
\end{Theo}

\begin{rem}
Note that in all the theorems one may construct $T^\lambda_{0,n}$ by selecting $T^\lambda_n\in\TT^\lambda_n$ for each $n$ and taking the product $T^\lambda_{0,n}=T^\lambda_n T^\lambda_{n-1} \cdots T^\lambda_0$.
In Theorem~\ref{Theo-4}~(i) we do not get existence of spectrum. Particularly, for Jacobi operators on the half line one finds such $u_n(\lambda)$ for $\lambda$ in the resolvent set where the measure $\mu$ is zero. 
But as $\mu=0$  is in fact an absolutely continuous measure, there is no contradiction here.
\end{rem}

\subsection{Some models with random decaying shell-matrix potential.}

As a first application for random models we will obtain a Last-Simon type of proof for absolutely continuous spectrum for a random decaying shell-matrix potential which couple an increasing number of wires.
By shell-matrix potential we mean a direct sum operator $V=\bigoplus_n V_n$ with $V_n\in \Her(\ell^2(S_n))$ being a Hermitian $s_n\times s_n$ matrix and $(V\psi)_n=V_n \psi_n$.
Let be given a sequence $(s_n)_{n=0}^\infty$ of natural numbers
with $s_{n+1}\geq s_n>0$ and let 
$$\GG_1\,:=\,\big\{(n,j)\,:\,n\in \NN_0\,,\, j\in\{1,\ldots,s_n\}\,\big\}\,,\quad S_n\,:=\,\{n\} \times \{1,\ldots,s_n\}\,.$$
Furthermore let be given a sequence $(a_j)_{j\in \NN}$ of real numbers with $\sup_j a_j - \inf_j a_j < 4$.
Then we define the free operator on $\ell^2(\GG_2)$ by
$$
(\Delta^{(1)} \psi)_{n,j}\,=\, -\psi_{n-1,j}\,-\,\psi_{n+1,j}\,+\,a_j \psi_{n,j}
$$
for $j\in\{1,\ldots, s_n\}$ where $\psi_{n-1,j}=0$ if $j>s_{n-1}$ or $n=0$ and $(\psi_{n,j})_{(n,j) \in \GG_1}\,\in\,\ell^2(\GG_1)$. This operators describes independent wires at different mean energies $a_j$ for the $j$-th wire. Moreover, the $j$-th wire starts at the $n$-th shell where $s_{n-1} < j \leq s_n$. After the coupling with $V_n$ this becomes a stair-like structure.
Let $V_n=V_n(\omega) \in \Her(s_n)$ be a  random Hermitian $s_n\times s_n$ matrix, that is a Hermitian $s_n\times s_n$ matrix valued random variable defined on some abstract probability space $(\Omega,\Aa,\PP)$.
We assume that  $(V_n)_n$ are independent and
\begin{equation} \label{eq-V-estimate}
\sum_{n=1}^\infty \left(\, \|\EE(V_n)\|+\EE(\|V_n\|^2+\|V_n\|^4)\,\right)\,<\,\infty\,
\end{equation}
where we use the standard matrix norm in any $\CC^{s_n \times s_n}$ and $\EE$ denotes the expectation value, that is $\EE f(V_n)=\int f(V_n(\omega)) d\PP(\omega)$.
Then, we consider the random operator
$$
(H^{(1)}_\omega \psi)_n\,=\,(\Delta^{(1)} \psi)_n\,+\,V_n(\omega)\,\psi_n
$$
where we identify $\psi_n=(\psi_{n,1},\ldots,\psi_{n,s_n})^\top \in \ell^2(S_n)$.
Let us consider
$$
I_1\,:=\,\bigcap_{j=1}^\infty (-2+a_j,2+a_j)
$$
which is not empty by the assumption on the $a_j$ above. In terms of a graph structure (applying an edge from $x$ to $y\neq x$ whenever $\langle \delta_x, H^{(1)}_\omega\,\delta_y\rangle \neq 0$) the graph looks stair-like, the horizontal lines come from the basic operator $\Delta^{(1)}$ and the vertical stripes symbol random graphs with random edge-weights coming from the $V_n$, where the 'strength' of $V_n$ is decaying.

\begin{center}
\includegraphics[width=7cm]{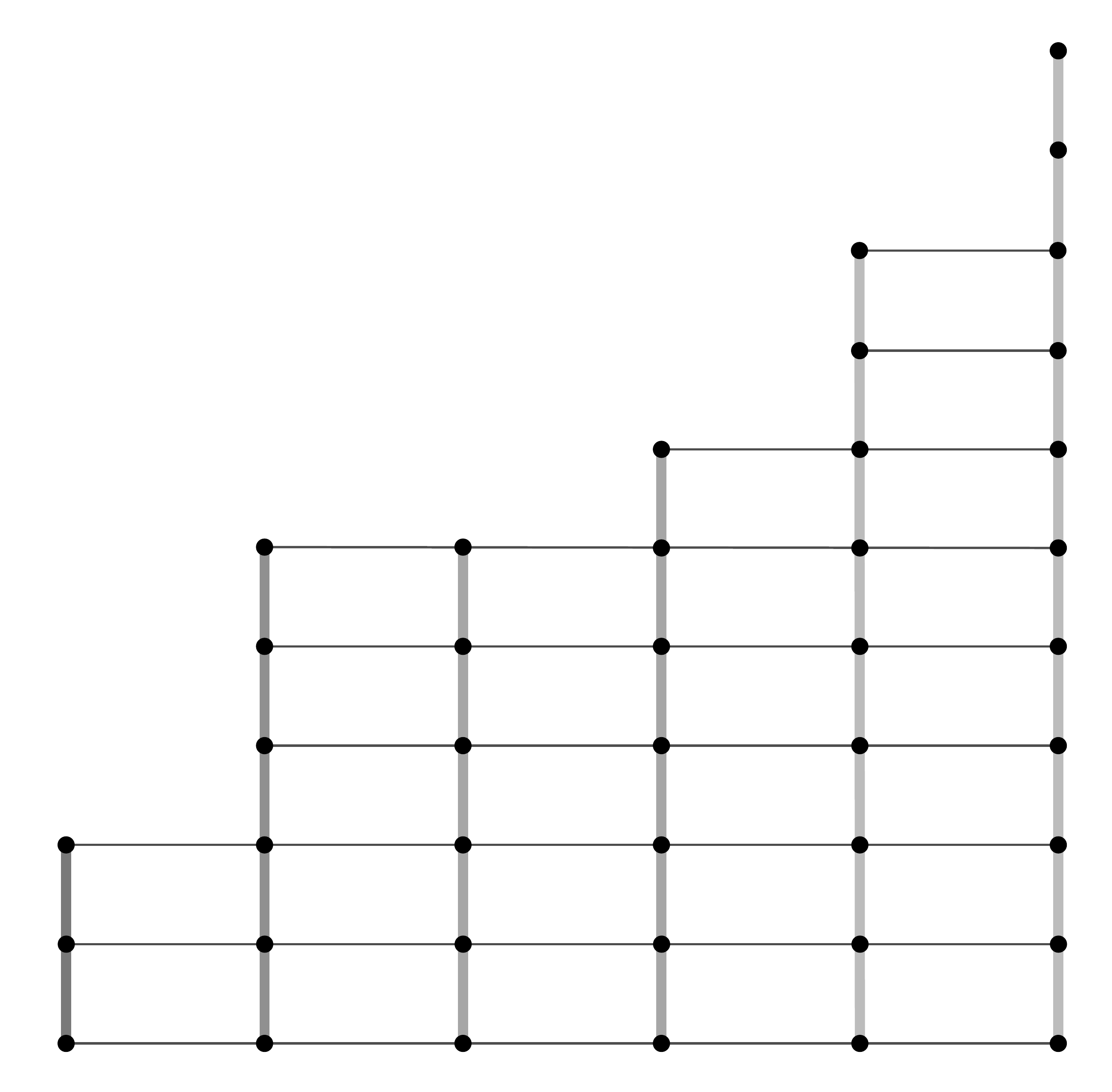}
\end{center}

\begin{Theo}\label{Theo-5}
Almost surely, the random operator $H^{(1)}_\omega$ has purely absolutely continuous spectrum in $I_2$  i.e. 
for $\PP$-almost all $\omega\in\Omega$, 
$$I_1\subset \spec_{ac}(H^{(1)}_\omega)\,,\quad
I_1\cap \spec_{pp}(H^{(1)}_\omega)=\emptyset\,,\quad
I_1\cap \spec_{sc}(H^{(1)}_\omega)=\emptyset\,$$
where $\spec_{ac}(H), \,\spec_{pp}(H),\,\spec_{sc}(H)$ denotes the absolutely continuous spectrum, the pure point spectrum, and the singular continuous spectrum of $H$, respectively (as closed sets).
\end{Theo}
\begin{rem}
Note, that in the non-random case ($V_n(\omega)=V_n$ is constant in $\omega$) the condition \eqref{eq-V-estimate} looks trace-class like, $\sum_n \|V_n\|<\infty$. However, since $s_n$ can grow towards infinity arbitrarily fast, it does not imply that the shell-matrix potential $V=\bigoplus_n V_n$ is actually trace class.
\end{rem}

With a unitary conjugation a similar result follows for the Bethe lattice. Therefore, let $\GG_2$ be the rooted binary Bethe lattice (or binary tree).
$\GG_2$ is constructed the following way: there is some root $0\in \GG_2$ being the zeroth generation and then each point in the $n$-th generation is connected to two 'children' in the $n+1$st generation. 

We let the $n$-th shell $S_n$ denote the set of points in the $n$-th generation (graph-distance $n$ from the root). Then, we have $s_n=\#S_n=2^n$ and we let $V_n=V_n(\omega) \in \Her(s_n)$ be a random
sequence of $s_n \times s_n$ Hermitian matrices as above satisfying \eqref{eq-V-estimate}.
Like above, $\psi_n=(\psi_{n,1},\ldots,\psi_{n,2^n})^\top \in \ell^2(S_n)=\CC^{2^n}$ is the part of $\psi\in \ell^2(\GG_2)$ in the $n$-th shell. We let $\Delta^{(2)}$ be the negative adjacency operator on $\ell^2(\GG_2)$ and define $H^{(2)}_\omega$ as above,
$$
(\Delta^{(2)} \psi)_{n,j}=-\psi_{n-1,\lceil j/2 \rceil}-\psi_{n+1,2j-1}-\psi_{n+1,2j}\,, \quad
(H^{(2)}_\omega \psi)_n\,=\,(\Delta^{(2)} \psi)_n\,+\, V_n(\omega) \psi_n\,.
$$
For $n=0$ we have again that $\psi_{n-1,j}=0$. With $\lceil j/2 \rceil$ we denote the smallest integer which is bigger or equal to $j/2$.
In terms of a graph structure it looks like this, where again the vertical stripes are random graphs describing $V_n$.
\begin{center}
\includegraphics[width=7cm]{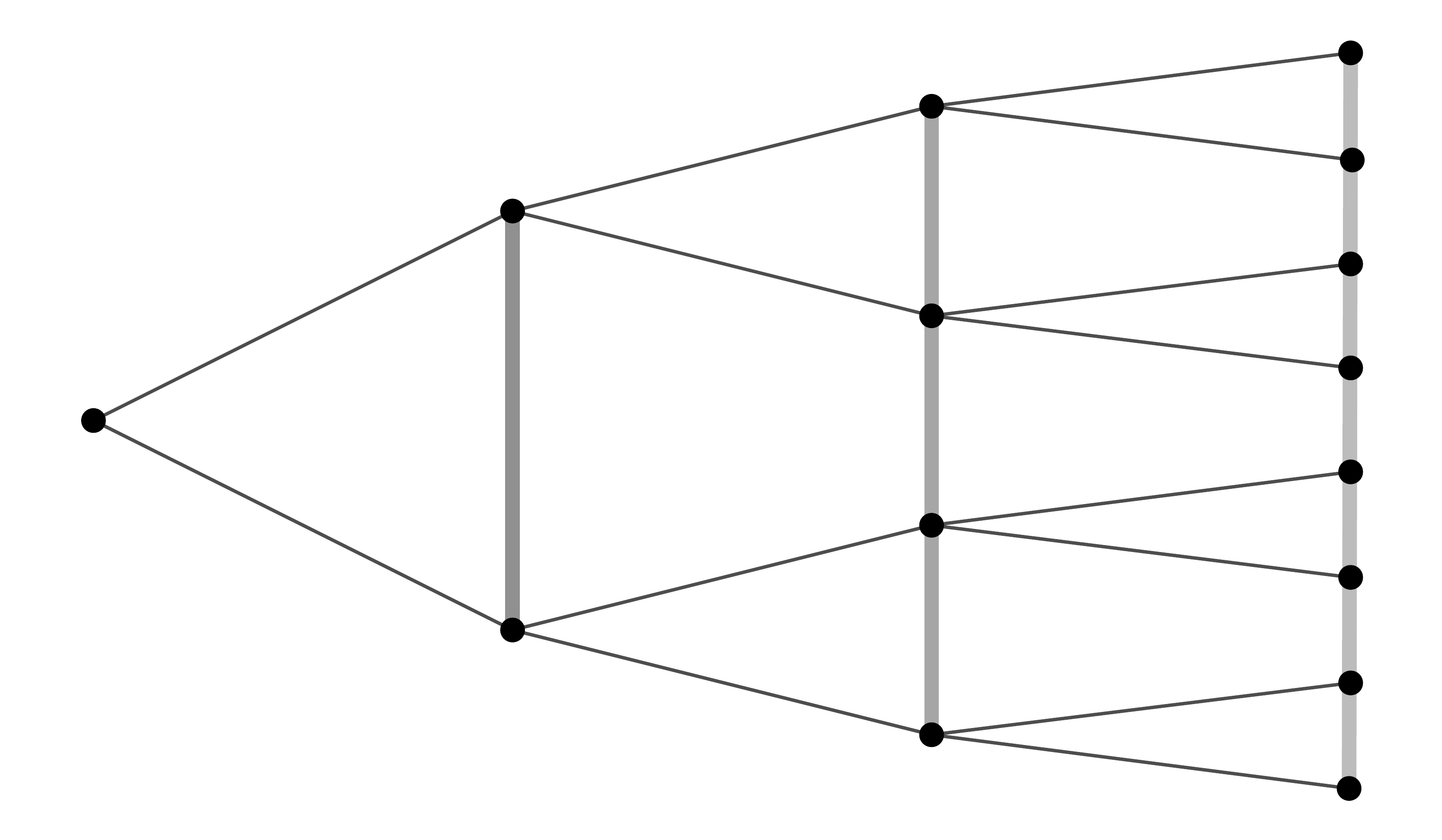}
\end{center}

The spectrum of $\Delta^{(2)}$ is absolutely continuous and given by the interval $[-2\sqrt{2}\,,\,2\sqrt{2}]$. In the interior we find absolutely continuous spectrum of $H^{(2)}_\omega$.
\begin{Theo}\label{Theo-6}
Almost surely, the random operator $H^{(2)}_\omega$ has purely absolutely continuous spectrum in $(-2\sqrt{2},2\sqrt{2})$, that is to say that almost surely,
$$
(-2\sqrt{2}\,,\, s\sqrt{2})\,\subset\,\spec_{ac}(H^{(2)}_\omega)$$
and
$$
\spec_{pp}(H^{(2)}_\omega) \cap (-2\sqrt{2}\,,\,2\sqrt{2})=\emptyset\,,\quad
\spec_{sc}(H^{(2)}_\omega)\cap (-2\sqrt{2}\,,\, 2 \sqrt{2}) = \emptyset$$
\end{Theo}

Another special case of Theorem~\ref{Theo-5} gives a result on random decaying matrix potentials on a strip. For uniform compactly supported matrix potentials this result has been already proved in \cite{FHS3}.
Therefore, let $\GG_3=\NN_0 \times S$ be the product graph of a finite graph $S$ and the half-line $\NN_0$, hence $\ell^2(\GG_3)=\ell^2(\NN_0)\otimes \ell^2(S)=\ell^2(\NN_0,\CC^s)$ and we choose $S_n=\{n\}\times S$ and may use $S=\{1,\ldots,s\}$. 
Then, a vector $\psi$ in $\ell^2(\GG_3)$ is given by a sequence $(\psi_n)_{n=0}^\infty$ with $\psi_n\in \ell^2(S)\cong \CC^s$ where $s=\#(S)$.
The unperturbed Laplacian on $\GG_3$ is given by
\begin{equation} \label{eq-H0}
(\Delta^{(3)} \psi)_n\,=\, -\psi_{n-1}-\psi_{n+1}+A\psi_n\,
\end{equation}
where formally $\psi_{-1}=0$ and $A\in \Her(s)$ is some basic Hermitian adjacency or Laplace operator on the graph $S$, possibly complex (with magnetic phases). 
Let $A$ have the eigenvalues $a_1,\ldots, a_s \in\RR$, then the spectrum of $\Delta^{(1)}$ is given by the union of $s$ bands:
$$
I\,:=\,\spec(\Delta^{(3)})\,=\,\bigcup_{j=1}^s [a_j-2, a_j+2]\,
$$
We will also consider the intersection of the open bands as before
$$
I_3\,:=\,\bigcap_{j=1}^s (a_j-2,a_j+2)\,=\,(-2+\max_j a_j, 2+\min_j a_j)\,.
$$
 Then, as before, we add a random, independent, decaying matrix potential $V_\omega=\bigoplus_n V_n(\omega)$ satisfying \eqref{eq-V-estimate} and consider the random operator
$H^{(3)}_\omega=\Delta^{(3)}+V_\omega$ given by
$$
(H^{(3)}_\omega \psi)_n\,=\,-\psi_{n-1}-\psi_{n+1}+A\psi_n+V_n(\omega) \psi_n\,.
$$
\begin{Theo}\label{Theo-7}
{\rm (i)} The random operator $H^{(3)}_\omega$ has, almost surely, 
purely absolutely continuous spectrum in the intersection of the open bands $I_3$, that is, for $\PP$-almost all $\omega\in \Omega$,
$$
I_3\subset \spec_{ac}(H^{(3)}_\omega)\,,\quad
I_3\cap \spec_{pp}(H^{(3)}_\omega)=\emptyset\,,\quad
I_3\cap \spec_{sc}(H^{(3)}_\omega)=\emptyset
$$
{\rm (ii)} Almost surely, $I=\spec_{ac}(H^{(3)}_\omega)=\spec_{ess}(H^{(3)}_\omega)$, that is, the absolutely continuous spectrum of $H^{(3)}_\omega$ is equal to the union of bands $I$ which is also equal to the essential spectrum of $H^{(3)}_\omega$, almost surely.
\end{Theo}
\begin{rem}
{\rm (i)} Note that the result is actually slightly more general than \cite{FHS3} as they have the additional assumptions that the distribution of all matrix potentials is supported in one compact set $K$.
Under these additional assumption, the bound $\sum_n \EE(\|V_n\|^4)<\infty$ follows automatically from
$\sum_n \EE(\|V_n\|^2)<\infty$.\\
{\rm (ii)} We do not prove pure absolutely continuous spectrum in part (ii).
As a set, the essential spectrum and the absolutely continuous spectrum are the same. We expect the absolutely continuous spectrum to be pure away from (internal and external) band edges. This will be dealt with elsewhere.
\end{rem}


\section{The algebra of the boundary resolvent data}

We want to consider the partial semi-group structure associated to the boundary resolvent data. For this reason we introduce the following binary operation between suitable matrices.
\begin{defini}\label{def-1}
\begin{enumerate}[{\rm (i)}]
\item We let $\Mm(q,r)$ denote
the set of $(q+r) \times (q+r)$ matrices $Q$ with the assigned $(q,r)$ splitting.
Such a matrix will be considered as a collection of 4 blocks (matrices) of size 
$q \times q$, $q\times r$, $r\times q$ and $r\times r$. This means,
$$
\Mm(q,r)\,=\,\left\{ Q=\pmat{\alpha & \beta \\ \gamma & \delta}\,\in\,\CC^{(q+r)\times (q+r)}\,:\, \alpha \in \CC^{q\times q},\, \beta \in \CC^{q\times r},\, \gamma \in \CC^{r\times q},\,\delta\in \CC^{r\times r}\right\}
$$
\item Let $Q\,\in\, \Mm(q,r)$ and $R\,\in\,\Mm(r,s)$ with the corresponding splittings
$$
Q=\pmat{\alpha & \beta \\ \gamma & \delta}\;,\quad R=\pmat{\tilde \alpha & \tilde \beta \\ \tilde \gamma & \tilde\delta}\;.
$$  
We say that the ordered pair $(Q,R)$ is $\trr_r$ suitable if $\II_r-\alpha\tilde \delta$ is invertible (in which case also $\II_r-\tilde \delta\alpha$ is invertible) where $\II_r$ denotes the $r\times r$ identity matrix. 

If $(Q,R)$ is $\trr_r$ suitable then we define 
$$
Q\trr_r R\,=\,
\pmat{ \alpha & \beta \\ \gamma & \delta}\;\trr_r\; \pmat{\tilde \alpha & \tilde \beta \\ \tilde \gamma & \tilde \delta}
\,:=\,
\pmat{\hat \alpha & \hat \beta \\ \hat \gamma & \hat \delta}\;\in\;\Mm(q,s)
$$
by
\begin{align} \label{eq-sg-1}
\hat \alpha\,&:=\, \alpha\,+\,\beta\,(\II-\tilde \alpha \delta)^{-1}\, \tilde\alpha\,\gamma\;,\quad &
\hat \beta\,:=\,\beta\,(\II-\tilde\alpha \delta)^{-1}\, \tilde \beta \\
\label{eq-sg-2}
\hat \delta\,&:=\,\tilde \delta\,+\,\tilde \gamma\,(\II- \delta \tilde \alpha)^{-1}\, \delta\,\tilde \beta\;, \quad &\hat \gamma\,:=\,\tilde \gamma\,(\II-\delta \tilde \alpha)^{-1}\, \gamma\;.
\end{align}
The index $r$ indicates that the operation $\triangleright_r$ does in fact depend on the splitting of the matrices $Q$ and $R$ and particularly on $r$, the size of the second assigned diagonal block of $Q$ or the first diagonal block of $R$. If the blocks are clear we may omit the index $r$. 
\item Furthermore, we define the following subset of $\Mm(q,r)$
$$
\Mm(q,r,+)\,:=\,\left\{Q=\pmat{\alpha & \beta \\ \gamma& \delta} \in \Mm(q,r)\,:\,  \Im(\alpha)>0,\, \Im(\delta)>0,\,\Im(Q)\geq 0 \right\}
$$
where $\Im(M)\,:=\,\frac{1}{2\imath} (M-M^*)$ denotes the imaginary part in the sense of $C^*$ algebra operations for square matrices $M$. Similarly, we let $\Mm(q,r,-)=-\Mm(q,r,+)$.
\end{enumerate}
\end{defini}

We use this operation for going forward along the shells in the graph.
The symbol $\trr$ shall indicate this radial move going from a lower shell in the graph outward, cf. Proposition~\ref{prop-bd}.

\begin{prop}\label{prop-sg}(Associativity)
\begin{enumerate}[{\rm a)}]
\item Let be given a $(q+r)\times (q+r)$ matrix $Q$, a $(r+s)\times(r+s)$ matrix $R$ and an $(s+t)\times (s+t)$ matrix $S$.
Then, if all the appearing combinations below are suitable, then
$$
(Q\,\trr_r R) \trr_s S\,=\,Q\,\trr_r\, (R\,\trr_s\, S) 
$$
and therefore we may define this to be $Q \trr_r R \trr_s S$. If the splitting of $Q,R,S$ is clear, we may simply write 
$Q\trr R \trr S$.
\item If $Q\in \Mm(q,r,\pm)$, $R\in \Mm(r,s,\pm)$ then $(Q,R)$ is $\trr_r$ suitable and $Q\trr_r R\in \Mm(q,s,\pm)$.
Particularly, $\big( \bigcup_{q,r} \Mm(q,r,\pm) ,\trr\,\big)$ are partial semi-groups.
\end{enumerate}
\end{prop}

Before coming to the proof of this proposition, let us state the main fact for this definition.

\begin{prop} \label{prop-bd}
For $z\in \CC$, $z\not \in \RR$ and all $m\leq n$ we have
$$
R^z_{m,n}\,\in\,\begin{cases}  \Mm(r_m, r_{n+1},+)\ & \text{if}\;\;\im(z)>0 \\ 
\Mm(r_m,r_{n+1},-) & \text{if}\;\;\im(z)<0 \end{cases}
$$ 
Furthermore,  for $l\leq m \leq n-1$, $z\in\CC$, $z\not \in A_{l,m} \cup A_{m+1,n} \cup B_{l,m,n}$ we have
$$
R^z_{l,n}\,=\,R^z_{l,m}\,\trr_{r_{m+1}}\, R^z_{m+1,n}\;.
$$
\end{prop}
 
The crucial parts of these propositions come from the following Lemma.
\begin{lemma}\label{lem-id-trr}
Let $\Gamma_1 \in \CC^{n_1 \times n_1}$ and $\Gamma_2\in \CC^{n_2\times n_2}$ be invertible matrices,
$\Upsilon\in \CC^{n_1\times q}$, $\Phi\in \CC^{n_1 \times r}$, $\tilde \Upsilon \in \CC^{n_2 \times r}$ and $\tilde \Phi \in \CC^{n_2 \times  s}$ such that
$$
Q\,=\,\pmat{\alpha & \beta \\ \gamma & \delta}= \pmat{\Upsilon^* \\ \Phi^*} \Gamma_1 \pmat{\Upsilon & \Phi}\;, \quad
R\,=\,  \pmat{\tilde \alpha & \tilde \beta \\ \tilde \gamma & \tilde \delta}=\pmat{\widetilde \Upsilon^* \\ \widetilde \Phi^* } \Gamma_2 \pmat{\widetilde \Upsilon & \widetilde \Phi},
$$
and let $(Q,R)$ be $\trr_r$ suitable, then
$$
Q \trr_r R\,=\,\pmat{ \Upsilon^* & \nul \\ \nul & \widetilde \Phi^*} \,
\pmat{\Gamma_1^{-1} & -\Phi \widetilde \Upsilon^* \\ -\widetilde \Upsilon \Phi^* & \Gamma_2^{-1}}^{-1}\,
\pmat{ \Upsilon & \nul \\ \nul & \widetilde \Phi }
$$
if the occurring inverse exists.
\end{lemma} 
\begin{proof}
Let
$$
\Gamma\,:=\,\pmat{\Gamma_1^{-1} & -\Phi \widetilde \Upsilon^* \\ -\widetilde \Upsilon \Phi^* & \Gamma_2^{-1}}^{-1}\,=\,
\pmat{A^{-1} & A^{-1}\Phi\widetilde\Upsilon^* \Gamma_2 \\ D^{-1}\widetilde \Upsilon\Phi^* \Gamma_1 & D^{-1}} 
$$
where $A,D$ denote the Schur complements
$$
A=\Gamma_1^{-1}\,-\,\Phi \widetilde \Upsilon^* \Gamma_2 \widetilde \Upsilon \Phi^* \,=\, \Gamma_1^{-1}\,-\,\Phi \,\tilde \alpha\, \Phi^*,\quad D=\Gamma_2^{-1}\,-\,\widetilde \Upsilon\, \delta\, \widetilde \Upsilon^*\;.
$$
Furthermore, let
$$
\pmat{ \Upsilon^* & \nul \\ \nul & \widetilde \Phi^*} \,\Gamma\, \pmat{ \Upsilon & \nul \\ \nul & \widetilde \Phi }\,=:\, \pmat{\hat \alpha & \hat \beta \\ \hat \gamma & \hat \delta}
$$
Using the resolvent identity $A^{-1}=\Gamma_1+A^{-1}\Phi\,\tilde \alpha\,\Phi^* \Gamma_1$ we obtain 
for $\zeta:=\Upsilon^*A^{-1}\Phi$,
$$
\zeta\,=\,\Upsilon^* A^{-1} \Phi\,=\,\beta\,+\,\zeta\,\tilde \alpha\,\delta \;\;\Rightarrow\;\;\zeta\,=\,\beta\,(\II-\tilde \alpha \delta)^{-1}
$$
where the inverses exist because $Q$ and $R$ are $r$-suitable. 
Then, 
$$
\hat \alpha\,=\,\pmat{\Upsilon \\ \nul}^*\Gamma \pmat{\Upsilon \\ \nul }\,=\,\Upsilon^* A^{-1} \Upsilon\,=\, \alpha\,+\,\zeta\,\tilde\alpha\,\gamma\,=\,\alpha\,+\,\beta\,(\II-\tilde\alpha \delta)^{-1}\,\tilde\alpha\,\gamma\;
$$
and furthermore
$$
\hat \beta\,=\,\Upsilon^* A^{-1} \Phi \widetilde \Upsilon^*\,\Gamma_2\,\widetilde \Phi\,=\,
\zeta\,\tilde\beta\,=\,\beta\,(\II-\tilde \alpha \delta)^{-1}\,\tilde \beta\;.
$$
These equations correspond exactly to \eqref{eq-sg-1}. Analogously, one obtains \eqref{eq-sg-2} as well.
\end{proof}

\begin{proof}[Proof of Proposition~\ref{prop-sg}]
Part a). We can use Lemma~\ref{lem-id-trr} twice in different ways to get  
$$
(Q\trr R) \trr S \,=\,
\pmat{\II_q & \nul & \nul \\ \nul & \nul & \II_t}
\pmat{Q^{-1} & \smat{\nul & \nul \\ -\II_r & \nul} \\ \smat{\nul & -\II_r \\ \nul & \nul} & R^{-1} & 
\smat{\nul & \nul \\ -\II_s & \nul} \\ & \smat{\nul & -\II_s \\ \nul & \nul} & S^{-1} }^{-1} \pmat{\II_q & \nul  \\ \nul & \nul \\ \nul & \II_t}\,=\, Q \trr (R \trr S)
$$
For the left equation we use
$\Gamma_1^{(l)} =\pmat{Q^{-1} & \smat{\nul & \nul \\ -\II_r & \nul} \\ \smat{\nul & -\II_r \\ \nul & \nul} & R^{-1} }^{-1}$ and $\Gamma_2^{(l)} = S$. For the right equation we use $\Gamma_1^{(r)}=Q$ and
$\Gamma_2^{(r)}=\pmat{R^{-1} & \smat{\nul & \nul \\ -\II_s & \nul} \\ \smat{\nul & -\II_s \\ \nul & \nul} & S^{-1} }^{-1}$ .
Formally, this calculation only proves the identity if all the appearing inverses and particularly $\Gamma_1^{(l)}$ and $\Gamma_2^{(r)}$  exist. By continuity we get the statement in general.

For part b) we use the same notations as above in Definition~\ref{def-1}. Note first that $\Im(\delta)>0$ and $\Im(\tilde \alpha)>0$ implies $\Im(\tilde \alpha^{-1}-\delta)<0$ and therefore $(\II-\tilde \alpha \delta)^{-1}=(\tilde \alpha^{-1}-\delta)^{-1} \tilde \alpha^{-1}$ exists.
This implies that $Q$ and $R$ are $r$-suitable.

Now, for the case $\Im(Q)>0,\,\Im(R)>0$ we can use Lemma~\ref{lem-id-trr} with
$\Gamma_1=Q$, $\Gamma_2=R$ to see immediately that $\Im(Q \trr R)>0$. 
By continuity we get $\Im(Q\trr R)\geq 0$ for $Q\in \Mm(q,r,+)$ and $R\in \Mm(r,s,+)$.
Consider the matrix 
$$
P:=\pmat{ \alpha & \beta \\ \gamma & \delta - \tilde \alpha^{-1}}\,=\,Q+\pmat{\nul & \nul \\ \nul & -\tilde \alpha^{-1}}\;.
$$
We claim $\Im(P)>0$. Indeed, $\Im(P)\geq 0$ is clear by $\Im(\tilde \alpha)>0$ and $\Im(Q)\geq 0$.
Assume
$$
v=\pmat{v_1 \\ v_2} \in \CC^{q+r}\,:\,
0\,=\,v^* \Im(P) v\,=\, v^* \Im(Q) v\,+\,
v_2^* \Im(-\tilde \alpha^{-1})\,v_2
$$
Because $\Im(Q)\geq 0$ and $\Im(-\tilde\alpha^{-1})>0$ this implies $v_2=0$ and thus $v_1^* \Im(\alpha) v_1\,=\,0$ which immediately gives $v_1=0$ as  $\Im(\alpha)>0$. Thus $v=0$ and $\Im(P)>0$.
Using the Schur complement formula and \eqref{eq-sg-1} we see that
$$
\hat \alpha\,=\,\left[\, \pmat{\II_q & \nul}\,P^{-1}\, \pmat{\II_q \\ \nul} \,\right]^{-1} \qtx{implying}
\Im(\hat \alpha)>0\;.
$$
Similarly one proves $\Im(\hat \delta)>0$.
\end{proof}
 
So finally we use Lemma~\ref{lem-id-trr} to prove the important Proposition~\ref{prop-bd}.
 
\begin{proof}[Proof of Proposition~\ref{prop-bd}]
The first statement is clear by definition as $\Upsilon_m$ and $\Phi_n$ are matrices of full rank equal to the number of columns. Only note that for $m=n$ the matrix $(\Upsilon_n \Phi_n)$ may not be of full rank and therefore the imaginary part of $\smat{\alpha^z_n & \beta^z_n \\ \gamma^z_n & \delta^z_n}$ may not be strictly positive.

Now, let $\Upsilon, \Phi$ denote the maps $\Upsilon_l$ and $\Phi_m$ with the natural extension of the  image set to $\ell^2(\GG_{l,m})$ which means
$\Upsilon \varphi = P_{l,m}^* P_l \Upsilon_l \varphi$ and $\Phi \psi = P_{l,m}^* P_m \Phi_m \psi$.
Similarly, we denote by $\hat \Upsilon, \hat \Phi$ the natural extensions of $\widetilde\Upsilon_{m+1}$ and $\widetilde\Phi_n$ with the image set extended to
$\ell^2(\GG_{m+1,n})$. Furthermore, let $\Gamma_1:=(H_{l,m}-z)^{-1}$,  $\Gamma_2:=(H_{m+1,n}-z)^{-1}$.
Then, we have exactly the situation as in Lemma~\ref{lem-id-trr} with $\Gamma=(H_{l,n}-z)^{-1}$ where $\Gamma$ is defined as in the proof of Lemma~\ref{lem-id-trr}. Thus, the lemma gives the result directly.
\end{proof}

\section{Transformation to matrix multiplication}

A simple, natural partial semi-group structure is given by matrix multiplication (of rectangular matrices). Thus,  it is a valid question whether one can transform the $\trr$ operation into a matrix multiplication. 
In fact, considering one-channel operators one has a particular transfer matrix structure coming from the resolvent boundary data.
The form of these transfer matrices leads to the definitions \eqref{eq-tr-1}, \eqref{eq-tr-2}, the only subtlety is that $\beta^z_n$ now is not necessarily a square matrix and the inverse of $\beta^z_n$ appearing in \cite{Sa-OC} makes no immediate sense. 
Still, working with right-inverses, Proposition~\ref{prop-1} can be seen as transforming the relation of Proposition~\ref{prop-bd} into multiplication of sets of matrices. 
The main focus of this section is the proof of Proposition~\ref{prop-1} and we will also obtain Proposition~\ref{prop-2}.

In analogy to the defined transfer matrices we may generally make the following definition.
\begin{defini}
\begin{enumerate}[{\rm (i)}]
\item For $q\leq r$ we denote by $\Mm_\trr(q,r)$ the set of $(q+r) \times (q+r)$ matrices with assigned $(q,r)$ splitting where the upper right $q \times r$ matrix part $\beta$ has full rank $q$ (that is to say that $\beta$ as a linear map is surjective), i.e.
$$
\Mm_\trr(q,r)\,:=\, \left\{ Q=\pmat{\alpha & \beta \\ \gamma & \delta}\,\in\,\Mm(q,r) \,:\, \beta \in \CC^{q\times r}=\Ll(\CC^r,\CC^q)\;\text{is surjective} \;\right\}\,.
$$
\item Let $Q = \pmat{\alpha & \beta \\ \gamma & \delta} \in \Mm_\trr(q,r)$  and
let $\BB=\{B\in \CC^{r \times q}\,:\, \beta\,B\,=\,\II_q\}$ denote the set of right-inverses to $\beta$. 
Furthermore let $\BBb\:=\{B\in \CC^{r \times q}\,:\, \beta\,B\,=\,\alpha\}$.
Then, we define the associated affine spaces of transfer-matrices 
$$
\TT_Q\,:=\, \left\{ \,\pmat{B & -\bbb \\ \delta B & \gamma-\delta \bbb }\,\in\,\CC^{2r \times 2q} \,:\, B\,\in\,\BB\,,
\;\bbb\,\in\,\BBb\,\right\}
$$
and
$$
\T_Q\,:=\, \left\{\,\pmat{B & -B\alpha \\ \delta B & \gamma-\delta B \alpha }\,\in\,\CC^{2r \times 2q} \,:\, B\,\in\,\BB\right\}\,\subset\,\TT_Q
$$
as well as the left and right-parts 
$$
\D_Q\,:=\,\left\{\pmat{B\\ \delta B}\,:\,B \in \BB \right\}
 \qtx{and} \N_Q\,:=\,\left\{\pmat{-\bbb \\ \gamma-\delta \bbb }\,:\,\bbb \in \BBb \right\}
$$
Note that $\D_Q=\TT_Q\smat{\II_q \\ \nul}=\T_Q\smat{\II_q \\ \nul}$ and
$\N_Q=\TT_Q \smat{\nul \\ \II_q}$. 
The used notation is based on '{\bf D}irichlet' and 'von {\bf N}eumann' boundary conditions.
\end{enumerate}
\end{defini}

\begin{rem}
If $\alpha$ is invertible, then $\BBb=\BB\,\alpha$ and one could replace $\bbb=B_2\alpha$ for some $B_2\in\BB$. But if $\alpha$ is 
not invertible, then one may only have $\BB\alpha\subset \BBb$.
\end{rem}

\begin{prop}\label{prop-lin}
Let $Q\in \Mm_\trr(q,r)$, $R\in \Mm_\trr(r,s)$, so particularly $q \leq r \leq s$.
Assume $Q$ and $R$ are $\trr_r$-suitable, then  $Q \trr R \in \Mm_\trr(q,s)$ and we have:
$$
\D_{Q\trr R}\,=\,\T_R\,\D_{Q}\,=\,\TT_R\,\D_Q\qtx{and}
\T_R\,\N_Q\,\subset\, \TT_R\,\N_Q\,=\, \N_{Q \trr R}\;.
$$
In particular we find
$\TT_R \TT_Q \,\subset\,\TT_{Q\trr R}\,.$
\end{prop}

\begin{proof}
First, from \eqref{eq-sg-1} one immediately sees that $\hat \beta$ is surjective and thus has maximal rank, so
$Q\trr R \in \Mm_\trr (q,s)$.
We use the same notations for $Q$ and $R$ as always (see Definition~\ref{def-1}.ii\;) 
and let
$$
\pmat{\hat B & -\hat \bbb \\ X & Y}\,:=\,\pmat{\tilde B & - \tilde \bbb \\ \tilde \delta \tilde B & \tilde\gamma - \tilde \delta \tilde \bbb}
\pmat{B & - \bbb \\ \delta B & \gamma - \delta \bbb}\;,\;\;
Q\trr R = :\pmat{\hat \alpha & \hat \beta \\ \hat \gamma & \hat \delta}
$$
where $\beta B=\II_q\;, \tilde \beta \tilde B = \II_r$, $\beta \bbb=\alpha$, $\tilde \beta \tilde \bbb = \tilde \alpha$.
Then, we need to show
$$
\hat \beta \hat B=\II_q\;,\quad 
X=\hat\delta \hat B\;, \quad
\hat \beta \hat \bbb=\hat \alpha\;,
\quad Y=\hat \gamma-\hat \delta \hat \bbb\;.
$$
First, we have 
$\hat B = (\tilde B - \tilde \bbb \delta) B$ and using \eqref{eq-sg-1} we find $\hat \beta \hat B_1=\II_q$.
Then, we obtain
$$
\hat \delta \hat B\,=\,\tilde \delta \hat B + \tilde \gamma (\II-\delta\tilde \alpha)^{-1} \delta (\II-\tilde \alpha \delta) B\,=\,\tilde \delta \hat B_1 + \tilde \gamma \delta B\,=\,X\;.
$$
So we get indeed that $\smat{\hat B \\ X} \in \D_{Q\trr R}$ and hence $\TT_R \D_Q\subset \D_{Q \trr R}$.
Clearly, $\T_R \D_Q \subset \TT_R \D_Q$ and using the special case $\tilde \bbb=\tilde B \alpha$ we get
$\hat B=\tilde B (\II-\tilde \alpha \delta) B$. By Proposition~\ref{prop-A1} we get all right inverses of $\hat \beta$ by varying $\tilde B$ and $B$ over all right inverses to $\tilde \beta$ and $\beta$. Hence, $\T_R \D_Q=\D_{Q\trr R}$ which finishes the proof of the first statement. $\hfill \blacksquare$.

Considering $\hat\bbb$, we get
\begin{equation} \label{eq-b1}
\hat\bbb= \tilde B \bbb +\tilde \bbb \gamma - \tilde\bbb \delta \bbb\,=\, (\tilde B-\tilde\bbb \delta) \bbb\,+\,\tilde \bbb \gamma
\end{equation} 
which gives
\begin{equation} 
\label{eq-b2}
\hat \beta \hat \bbb\,=\,
\beta\,(\II-\tilde \alpha \delta)^{-1} \tilde \beta\hat \bbb\,=\,
\beta\,(\II-\tilde\alpha \delta)^{-1}\left[ (\II-\tilde \alpha \delta)\bbb\,+\,\tilde \alpha \gamma\right]
\,=\,
\alpha+\beta(\II-\tilde \alpha \delta)^{-1} \tilde \alpha \gamma\,=\,\hat \alpha\;.
\end{equation}
Finally, we need to check that $Y=\hat\gamma-\hat \delta\,\hat\bbb$\;. We have that
\begin{equation}
\label{eq-d1} 
Y\,=\,-\tilde \delta\,\hat\bbb\,+\,\tilde \gamma\,(\gamma-\delta \bbb)\;.
\end{equation}
and
\begin{equation}
\label{eq-d2} \hat \gamma - \hat \delta \,\hat \bbb\,=\,
\hat \gamma\,-\,\tilde \delta \hat \bbb\,-\, \tilde \gamma\, (\II-\delta \tilde \alpha)^{-1} \delta \tilde \beta\,\hat \bbb\;.
\end{equation}
So we see that the left hand sides of \eqref{eq-d1} and \eqref{eq-d2} are equal if and only if
\begin{equation}
\label{eq-d3} \tilde \gamma\,(\gamma-\delta \bbb)\,=\,
\tilde \gamma\,(\II-\delta \tilde \alpha)^{-1}\,\left[\gamma\,-\,\delta \tilde \beta\,\hat \bbb \right]\,
\end{equation} 
where we used $\hat \gamma = \tilde \gamma (\II-\delta \tilde \alpha)^{-1} \gamma$.
Let us transform the right hand side,
\begin{align} 
\gamma-\delta \tilde \beta\hat \bbb\,&=\,
\gamma\,-\,\delta \tilde \beta \,\left( (\tilde B-\tilde \bbb \delta)\bbb +\tilde \bbb \gamma \right)
 \,=\, \label{eq-d4}
 (\II-\delta\tilde \alpha) \gamma\,-\,
(\delta-\delta \tilde \alpha \delta) \bbb \;.
\end{align}
Using \eqref{eq-d4}  the right hand side of \eqref{eq-d3} becomes
$\tilde \gamma (\gamma-\delta\bbb) $ which is equal to the left hand side of \eqref{eq-d3}.
Thus, we proved $\smat{-\hat\bbb\\ Y} \in \N_{Q\trr R}$ and 
$\T_R\N_Q \subset \TT_R \N_Q\subset \N_{Q\trr R}$. $\hfill \blacksquare$\\
For equality in the second inclusion we first claim the following:
For any fixed $\tilde \bbb$ such that $\tilde\beta \tilde \bbb=\tilde \alpha$, for $\tilde \BB=\{\tilde B\,:\,\tilde \beta \tilde B=\II\}\,,\; \BBb=\{\bbb\,:\,\beta \bbb=\alpha\}$
 we get
$$
(\tilde \BB-\tilde \bbb \delta) \BBb\,=\,\{ B\in \CC^{s \times q}\,:\,\hat \beta B=\alpha\}\;.
$$
First, because $\II-\tilde \alpha\delta$ is invertible, we get 
$\tilde \BB-\tilde \bbb\delta = \tilde \BB(\II-\tilde \alpha \delta)$ and therefore,  by Proposition~\ref{prop-A1}
$(\tilde \BB-\tilde \bbb \delta) \BB (\alpha+\varepsilon\II)\,=\, \tilde \BB(\II-\tilde \alpha \delta)\BB (\alpha+\varepsilon\II)=
\hat \BB (\alpha+\varepsilon\II)$
where $\hat \BB=\{\hat B\,:\,\hat \beta \hat B\,=\,\II\}$. 
The limit $\varepsilon\to 0$ gives the claim.\footnote{Note that for small but non-zero $\varepsilon$, $\alpha+\varepsilon\II$ is invertible.
If $\alpha$ is invertible, then $\BB\alpha=\BBb$ and 
$\hat\BB\alpha\,=\,\{ B\in \CC^{s \times q}\,:\,\hat \beta B=\alpha\}$ and we do not need this trick $\varepsilon\to 0$, however, if $\alpha$ is not invertible this limit may give a bigger (higher dimensional) affine space as using $\hat\BB\alpha$.}

Finally, for any element $\hat \bbb_0$ such that $\hat \beta \hat \bbb_0= \hat \alpha$ and some fixed $\tilde \bbb$ as above, we have
$$
\hat \beta (\hat \bbb_0-\tilde \bbb \gamma)\,=\,\hat \alpha-\beta(\II-\tilde \alpha \delta)^{-1}\tilde \alpha \gamma\,=\,\alpha\;.
$$
Thus, we find $\tilde B\in \tilde \BB$ and $\bbb\in\BBb$ such that
$$
\hat \bbb_0-\tilde \bbb \gamma\,=\,(\tilde B-\tilde \bbb \delta)\bbb\quad\Rightarrow\quad
\hat\bbb_0\,=\,(\tilde B-\tilde \bbb \delta)\bbb\,=\,\tilde \bbb \gamma\;.
$$
and therefore, $\pmat{-\hat\bbb_0 \\ \hat\gamma-\hat \delta \hat \bbb_0}\,\in\, \TT_R \N_Q$ and we get
$\N_{Q\trr R}\subset \TT_R \N_Q$\;. But in fact, we get the full set $\N_{Q\trr R}$ just by varying $\tilde B$ and $\bbb$ for any fixed $B$ and $\tilde \bbb$.
\end{proof}

Now we obtain Proposition~\ref{prop-1} basically as a corollary:

\begin{proof}[Proof of Proposition~\ref{prop-1}]
First, if $z \not \in A_{m,n} \cup A_{m+1,l} \cup B_{l,m,n}$ then $R^z_{l,m}\in \Mm_\trr(r_l,r_{m+1})$ and $R^z_{m+1,n}\in \Mm_\trr(r_{m+1},r_{n+1})$ are $\trr_{r_{m+1}}$ suitable and one can define 
$R^z_{l,m}$ by $R^z_{l,m} \trr R^z_{m+1,n}\in \Mm_\trr(r_l,r_{n+1})$, so $R^z_{l,m}$ is well defined at least after analytic continuation. Moreover, $\beta^z_{l,n}$ is of full rank (cf. \eqref{eq-sg-1}, hence, $z\not \in A_{l,n}$.
This particularly gives $A_{l,n}\subset A_{l,n-1} \cup A_{n} \cup B_{l,n-1,n}$. Since $A_n$ is finite and $B_{l,m,n}$ is always finite we get by induction that $A_{l,n}$ is finite for any $n\geq l$ (note $A_{l,l}=A_l$) proving part (i).

Proposition~\ref{prop-bd} and Proposition~\ref{prop-lin} give part (ii) as
$\TT^z_{m,n}=\TT_{R^z_{m,n}}$ and $\T^z_{m,n}=\T_{R^z_{m,n}}$.
Using $l=0$, $r_0=1$ part (iii) follows from (ii) as
$\D^z_n=\TT^z_{0,n}\smat{1\\0}=\T^z_{0,n}\smat{1\\0}$ and
$\N^z_n=\TT^z_{0,n} \smat{0\\1}$. Part (iv) follows from (iii) by induction.
\end{proof}

Proposition~\ref{prop-2} follows directly from the following.

\begin{prop}[Symplectic structure] Let $Q\in\Mm_\trr(q,r)$, $q\leq r$.\\
{\rm (i)} If $Q$ is Hermitian, $Q=Q^*$, then $\TT_Q \subset \HSP$, that means $\TT_Q$ is a set of (rectangular) Hermitian-symplectic matrices.
Moreover, for any $T_1, T_2 \in \TT_Q$ we find
$T_1^* J_r T_2=J_q$.\\
{\rm (ii)} If $Q$ is symmetric (in matrix sense), $Q=Q^\top$, then
$\TT_Q\subset \SP$, that means, $\TT_Q$ is a set of (rectangular) symplectic matrices. Moreover, for any $T_1, T_2 \in \TT_Q$ we find
$T_1^\top J_r T_2=J_q$.
\end{prop}
Note that generally $\TT_Q$ are sets of rectangular matrices, so we do not talk about elements of the Hermitian-symplectic or symplectic group, we rather talk about elements in the partial semi-groups as defined in Definition~\ref{def-HSP}.
\begin{proof}
Using our standard notations, $Q=Q^*$ implies $\alpha=\alpha^*$, $\beta=\gamma^*$ and $\delta=\delta^*$. Let
$T_i=\smat{B_i & -\bbb_i \\ \delta B_i & \beta^*-\delta \bbb_i}\in \TT_Q$, we find
$$
T_1^* J_r T_2\,=\, \pmat{B_1^* \delta & -B_1^* \\ \beta-\bbb_1^* \delta & \bbb_1^*}
\pmat{B_2 & -\bbb_2 \\ \delta B_2 & \beta^*-\delta \bbb_2}\,=\,
\pmat{\nul & -B_1^* \beta^* \\ \beta B_2 & -\beta \bbb_2+\bbb_1^* \beta^*}\,=\,J_q
$$
where we used $\beta B_i =\II_q=B_i^* \beta^*$ and $\beta \bbb_i=\alpha=\alpha^*=\bbb_i^*\beta^*$.
Using $T=T_1=T_2$ this shows $T\in \HSP$. 
The second part,starting with $Q=Q^\top$, follows by the same calculations replacing $*$ by the transpose $\top$ everywhere.
\end{proof}

\section{Relation to formal solutions of the eigenvalue equation}

The main point of this section is to prove Proposition~\ref{prop-3}.
The first part, Proposition~\ref{prop-3}~(i) states the following:
If we have for $z, \bar z \not \in \bigcup_{n=0}^\infty A_n $ a formal solution of 
$H\Psi=z\Psi+v P_0 \Upsilon_0$, then, there exist $T^z_n \in \TT^z_n$, $n\in \NN_0$, such that
$$
\pmat{\Upsilon_{1}^* \Psi_{1} \\ \Phi_0^* \Psi_0 }\,=\, T^z_0 \pmat{\Upsilon_0^* \Psi_0 \\ v }\qtx{and}
\pmat{\Upsilon_{n+1}^* \Psi_{n+1} \\ \Phi_n^* \Psi_n}\,=\,T^z_n \,\pmat{\Upsilon_{n}^* \Psi_{n} \\ \Phi_{n-1}^* \Psi_{n-1}}\,.
$$
In the second part it is stated that for any solution of the transfer matrix equation with any choice of transfer matrices there is also a corresponding formal solution $\Psi$ satisfying the equations above.
Let us start with the following lemma.
\begin{lemma} \label{lem-no-ker}
Let $z, \bar z \not \in A_n$, then for any $T^z_n\in \TT^z_n$ we have $\dim(\ker T^z_n))=0$.
\end{lemma}
\begin{proof} 
For $z\not \in A_n$ we have that $\TT^z_n$ exists. Hence, let $B\in \BB^z_n,\,\bbb \in \BBb^z_n$ and let
$$
\nul\,=\,T^z_n \pmat{v \\ w} \,=\,\pmat{B & -\bbb \\ \delta^z_n B & \gamma^z_n - \delta^z_n \bbb} \pmat{v \\ w}\,=\,\pmat{Bv-\bbb w \\ \delta^z_n Bv +(\gamma^z_n-\delta^z_n B)w}\,.
$$
Then, $Bv=\bbb w$ implies $\delta^z_n Bv=\delta^z_n \bbb w$ so that the second equation gives $\gamma^z_n w=0$. For $\bar z\not\in A_n$ the kernel of $\gamma^z_n=(\beta^{\bar z}_n)^*$ is just the zero vector so that $w=\nul$ 
which implies $v=\beta^z_n Bv = \nul$ and finally $\smat{v\\w}=\nul$. Hence, the kernel of $T^z_n$ consists of only the zero vector.
\end{proof}

\begin{proof}[Proof of Proposition~\ref{prop-3}]
Let $\Psi=(\Psi_n)_{n=0}^\infty$ be a formal solution of $H\Psi=z\Psi+v P_0\Upsilon_0$.
Moreover, let $\bfu_{n}\,:=\,\Upsilon_{n}^* \Psi_{n} \in \CC^{r_n} \,,\,\bfv_n:=\Phi_n^* \Psi_n\in \CC^{r_{n+1}}$, $n\geq 0$, and $\bfv_{-1}\,:=\,v$. 
Under the assumption that $\smat{\bfu_n \\ \bfv_{n-1}} \neq \nul$ (meaning it is not the zero vector) we will prove the existence of $T^z_n \in \TT^z_n$ such that
$\smat{\bfu_{n+1} \\ \bfv_n}=T^z_n \smat{\bfu_n \\ \bfv_{n-1}}$. For $n=0$ this is fulfilled by assumption and by Lemma~\ref{lem-no-ker} we obtain that
$\smat{\bfu_{n+1} \\ \bfv_{n}} \neq \nul$. Hence the existence of $T^z_n$ for any $n\geq 0$ then follows by induction.

So let $\smat{\bfu_n \\ \bfv_{n-1}}\neq \nul$. By the assumptions together with \eqref{eq-H} we have
\begin{equation}\label{eq-formal-sol}
(V_n-z)\,\Psi_n\,=\, \Upsilon_n \,\bfv_{n-1}\,+\,\Phi_n \,\bfu_{n+1}\qtx{for all} n\in \NN_0
\end{equation}
Assume for a moment that $V_n-z\II$ is invertible. 
Then by applying the inverse $(V_n-z)^{-1}$ on both sides and multiplying by $\Upsilon_n^*$ or $\Phi_n^*$ from the left we obtain
\begin{equation} \label{eq-formal-2}
\pmat{\bfu_n \\ \bfv_n}\,=\,R^z_n \,\pmat{\bfv_{n-1} \\ \bfu_{n+1}}\,=\,\pmat{ \alpha^z_n\,\bfv_{n-1}\,+\,\beta^z_n\,\bfu_{n+1} \\ \gamma^z_n\,\bfv_{n-1}\,+\,\delta^z_n \,\bfu_{n+1}}\;.
\end{equation}
As $z\not \in \A_n$ we have that $R^z_n$ is defined at least by analytic continuation, so the kernel of $(V_n-z)$ is orthogonal to the column vectors of $\Upsilon_n$ and $\Phi_n$. Hence, one may project
$\Psi_n$  to the subspace in $\ell^2(S_n)$ orthogonal to the kernel of $V_n-z\II$. Then, applying the inverse to \eqref{eq-formal-sol} in this subspace also gives \eqref{eq-formal-2}.  

Now take some $B_0\in \BB^z_n$ and let $\bfw\,:=\, \bfu_{n+1}\,-\,B_0(\bfu_n-\alpha^z_n \bfv_{n-1}) \in \CC^{r_{n+1}}$.
By the first equation of \eqref{eq-formal-2} we get $\beta^z_n \bfw=\nul$, thus, $\bfw\in \ker(\beta^z_n)$.
Let $\KK^z_n\:=\,\ker(\beta^z_n)^{r_n}$ denote the set of $r_{n+1} \times r_{n}$ matrices where each column vector is in the kernel of $\beta^z_n$.
Now, if $\bfu_n\neq \nul$, then at least one entry is not zero and one can create a matrix $K_1 \in \KK^z_n$ which has all column vectors equal to zero except for one, such that
$K_1\bfu_n=\bfw$, moreover, we let $K_2=\nul$. \\
If $\bfu_n=\nul$ then by assumption $\bfv_{n-1} \neq 0$ and we find $K_2\in \KK^z_n$ such that $K_2 \bfv_{n-1}\,=\,\bfw$ and we let $K_1=\nul$ in this case. Then, in either of the cases, we find for $B=B_0+K_1$ and $\bbb=B_0\alpha^z_n-K_2)$ that
$$
B\,\bfu_n\,-\,\bbb\,\bfv_{n-1}\,=\,\bfu_{n+1}\;.
$$
Plugging this into the second equation of \eqref{eq-formal-2} gives
$$
\bfv_n\,=\,\gamma^z_n \bfv_{n-1} \,+\,\delta^z_n(B\,\bfu_n\,-\,\bbb\,\bfv_{n-1})\,=\, \delta^z_n\, B\,\bfu_n\,+\,(\gamma^z_n-\delta^z_n \bbb)\,\bfv_n\;.
$$
Together, this means
$$
\pmat{ \bfu_{n+1} \\ \bfv_n}\,=\,
T^z_n\,\pmat{\bfu_n \\ \bfv_{n-1}} \qtx{where} T^z_n=\pmat{B & -\bbb \\ \delta^z_n B & \gamma^z_n - \delta^z_n \bbb}\,\in\,\TT^z_n\;. 
$$
$\hfill \blacksquare$

In part (ii) of Proposition~\ref{prop-3} we have the reversed situation: For some $z\not\in \bigcup_{n=0}^\infty A_n$ we are given some transfer matrices 
$T^z_n\in \TT^z_n$ and a starting vector $\smat{u \\ v}=:\smat{\bfu_0 \\ \bfv_{-1}}$ and we want to construct a
corresponding formal solution $\Psi=(\Psi_n)_{n=0}^\infty$ to $H\Psi=z\Psi+vP_0\Upsilon_0$.
Therefore, define
$$
\pmat{\bfu_{n+1} \\ \bfv_n}\,=\,T^z_n T^z_{n-1} \cdots T^z_0 \pmat{u \\ v} \qtx{and let} \pmat{\bfu_0 \\ \bfv_{-1}}\,:=\,\pmat{u \\ v}\;.
$$
It easy to check that $(\bfu_n, \bfv_{n-1})_{n \in \NN_0}$ satisfies the equations \eqref{eq-formal-2}.
If $V_n-z\II$ is invertible, then we define
$$
\Psi_n\,:=\,(V_n-z\II)^{-1}\,\left(\Upsilon_n \bfv_{n-1}+\Phi_n \bfu_{n+1} \right)
$$
for any $n\in \NN_0$.
If $V_n-z$ is not invertible, then 
all the column vectors of $\Upsilon_n$ and $\Phi_n$ are orthogonal to the kernel of $V_n-z\II$ 
because $R^z_n$ exists by analytic extension. Hence, 
we can restrict to the orthogonal complement of the kernel and take the inverse there. This way, $\Psi_n$ is always well defined.
It is clear that $(\Psi_n)_{n=0}^\infty$ satisfies \eqref{eq-formal-sol} and with \eqref{eq-formal-2} we find $\bfu_n=\Upsilon_n^*\Psi_n$ and $\bfv_n=\Phi_n^* \Psi_n$ for $n\geq 0$. Both together give that
$\Psi=\bigoplus_n \Psi_n=(\Psi_n)_n$ is indeed a formal solution of $H\Psi=z\Psi+vP_0\Upsilon_0$.
\end{proof}

\section{Weyl discs and limit point property}

In this part we prove Theorem~\ref{Theo-2} and prepare for the spectral averaging formula leading to Theorem~\ref{Theo-1}. 
Let us first introduce some notations.
The set of unitary $n\times n$ matrices will be denoted by $\UU(n)$, the disc of matrices with singular values smaller equal to one by $\DD(n)$, the set of Hermitian matrices by ${\rm Her}(n)$, the set of matrices with non-negative imaginary part by ${\rm Mat}_+(n)$ and ${\rm Mat}_-(n)=-{\rm Mat}_+(n)$. 
This means
$$
\UU(n)\,:=\,\{U\,\in\,\CC^{n\times n}\,:\,U^*U=\II\}\,,\quad
\DD(n)\,:=\,\{R\,\in\,\CC^{n\times n}\,:\,R^*R\leq\II\}
$$
$$
{\rm Her}(n)\,:=\,\{A\,\in\,\CC^{n\times n}\,:\,A^*=A \}\,,\quad
{\rm Mat}_\pm(n)\,:=\,\{A\,\in\,\CC^{n\times n}\,:\,\pm\Im(A)\geq \nul\}
$$
Recall that $\Im(A)=(A-A^*)/(2\imath)$ is the imaginary part in the sense of $C^*$ algebras, meaning that real and imaginary part are Hermitian matrices.

Let us consider the partial operators $H_{0,n}$ on $\ell^2(\GG_{0,n})$. We will view $\Upsilon_0$ and the $r_{n+1}$ column vectors of $\Phi_n$ as vectors in $\ell^2(\GG_{0,n})$ and simply write 
$\hat \Upsilon_0:=P_{0,n}^* P_0 \Upsilon_0$ and $\hat \Phi_n:=P_{0,n}^*P_n \Phi_n$.
Recall that with this convention we have
$$
R^z_{0,n}=\pmat{\alpha^z_{0,n} & \beta^z_{0,n} \\ \gamma^z_{0,n} & \delta^z_{0,n}}\,=\,
\pmat{\hat \Upsilon_0^* \\ \hat \Phi_n^*}\, (H_{0,n}-z)^{-1}\,\pmat{\hat \Upsilon_0 & \hat \Phi_n}
$$
For $z\in \CC^+$, meaning $\im(z)>0$, and $A\in {Mat}_+(r_n)$ let us define
$$
R^{z,A}_{0,n}=\pmat{\alpha^{z,A}_{0,n} & \beta^{z,A}_{0,n} \\ \gamma^{z,A}_{0,n} & \delta^{z,A}_{0,n}}\,=\,
\pmat{\hat \Upsilon_0^* \\ \hat \Phi_n^*}\, (H_{0,n}-\hat \Phi_n A \hat\Phi_n^*-z)^{-1}\,\pmat{\hat \Upsilon_0 & \hat\Phi_n}\,.
$$
Note that $\alpha^{z,A}_{0,n}$ are $1\times 1$ matrices and hence, just numbers in the upper half plane $\CC^+$ for $\im(z)>0$.
In analogy to Weyl circle and Weyl disc for Jacobi operators we define the following.

\begin{defini}
For $\im(z)>0$ the {\bf $\mathbf{n}$-th Weyl region} $\breve \WW_n^z\subset \CC^+$ and the {\bf $\mathbf{n}$-th Weyl disc} $\WW^z_n$ are defined by
$$
\breve{\WW}^z_n\,:=\,{\rm cl}\,{\{\alpha^{z,A}_{0,n}\,:\, A\in{\rm Her}(r_{n+1})\,\}}\qtx{and}
\WW^z_n\,:=\, {\rm cl}\,{\{\alpha^{z,A}_{0,n}\,:\,A\in{\rm Mat}_+(r_{n+1})\}}\;.
$$
where ${\rm cl}\,\Ss$ denotes the closure of a set $\Ss$ (this amounts to formally allowing $A$ to have infinite values.)
\end{defini}

The standard resolvent identity $(H-\Aa-z)^{-1}=(H-z)^{-1}+(H-z)^{-1} \Aa (H-\Aa-z)^{-1}$ yields
\begin{align}
\gamma^{z,A}_{0,n} \,&=\, \gamma^z_{0,n}\,+\,\delta^z_{0,n} \,A\, \gamma^{z,A}_{0,n} \quad\Rightarrow\quad
\gamma^{z,A}_{0,n}\,=\,(\II-\delta^z_{0,n} A)^{-1} \gamma^z_{0,n} \label{eq-gamma-A} \\
\alpha^{z,A}_{0,n}\,&=\,\alpha^z_{0,n}\,+\,\beta^z_{0,n} \,A\, \gamma^{z,A}_{0,n}\,=\, \alpha^z_{0,n}\,+\,\beta^z_{0,n} \,A(\II-\delta^z_{0,n}A)^{-1}\,\gamma^z_{0,n}\;. \label{eq-alpha-A}
\end{align}
Since $\Im(\delta^z_{0,n})>0$, the inverse $(\II-\delta^z_{0,n} A)^{-1}$ is well defined for $\Im(A)\geq 0$.\footnote{Indeed $\Im(\delta)>0, \Im(A)\geq 0$ and assuming $\delta A \psi=\psi$ we get $(A\psi)^* \delta A \psi=\psi^* A^* \psi$. The imaginary part of the right hand side is lower or equal to zero, and of the left hand side it is larger than zero, unless $A\psi=0$. Hence, $A\psi=0$ in which case $\psi=\delta A\psi=0$.}
Also, we see that the Weyl disc as defined here is the same as in Theorem~\ref{Theo-2}.

\begin{prop}\label{prop-W1}
Let $\im(z)>0$, then, we have the following: \\
{\rm (i)}
$\WW^z_n$ is a compact, closed disc in the upper half plane. If $r_{n+1}=1$ then
$\breve \WW^z_n=\partial \WW^z_n$ is the surrounding circle, if $r_{n+1}\geq 2$, then
$\breve \WW^z_n=\WW^z_n$.\\
{\rm (ii)} Let $z,\bar z \not \in \bigcup_{k=0}^\infty A_k$ and let $\Psi^z \subset \CC^G$ be the set of all formal\footnote{By formal solutions we mean that we consider the space of all $\GG$-sequences $\CC^\GG$ and not only $\ell^2$.} solutions $\psi=\bigoplus_{k=0}^\infty \psi_k$ of $H\psi=z\psi$ with $\Upsilon_0^* \psi_0=1$. 
Then, the radius $r_{z,n}$ of the Weyl disc $\WW^z_n$ satisfies 
$$
\frac{1}{4\,\im(z)^2\,r^2_{z,n}}\,=\,{\left( \min_{\psi^z\in \Psi^z } \sum_{k=0}^n \| \psi^z_k \|^2\right) \left( \min_{\psi^{\bar z}\in \Psi^{\bar z} } \sum_{k=0}^n \| \psi^{\bar z}_k \|^2\right) }\,.
$$
{\rm (iii)} $\WW^z_{n+1} \subset \WW^z_n$.
\end{prop}
\begin{proof}
Let $\breve \Ab^z_n:={\rm cl}\,{\{ (B-\delta^z_{0,n})^{-1}\,:\, B=B^*\}}$, 
$\Ab^z_n:={\rm cl}\,{\{ (B-\delta^z_{0,n})^{-1}\,:\, \Im(B)\leq 0\}}$,
and, moreover $\hat\Ab^z_n:=\{(B-\delta^z_{0,n})^{-1}\,:\,\Im(B)<0\}$. Note, $\Im(A)>0$ gives $\Im(A^{-1})<0$ and
$A(\II-\delta^z_{0,n}A)^{-1}=(A^{-1}-\delta^z_{0,n})^{-1}$ when $A$ is invertible. Using continuity and density of invertible matrices it is sufficient to consider invertible matrices $A$ in the definition of $\WW^z_n$ and $\widehat \WW^z_n$ and by \eqref{eq-alpha-A} one finds
$$
\breve \WW^z_n\,=\,\alpha^z_{0,n}\,+\,\beta^z_{0,n} \,\breve \Ab^z_n\, \gamma^z_{0,n}\;,\quad
\WW^z_n\,=\,\alpha^z_{0,n}\,+\,\beta^z_{0,n} \,\Ab^z_n\, \gamma^z_{0,n}\;,\quad
$$ 
Thus, $\breve \WW^z_n$ and $\WW^z_n$ are affine projections of $\breve \Ab^z_n$ and $\Ab^z_n$ under the same affine map. 
Note that
$$
A(\II-\delta^z_{0,n} A)^{-1}\,=\,\pmat{\II & \nul \\ -\delta^z_{0,n} & \II}\,\cdot\, A
$$
can be regarded as a generalized Möbius transform acting on ${\rm Mat}_+(r_n)$. Similarly as the Möbius transforms in the upper half plane it maps
generalized 'circles' and 'discs' into such 'circles' and 'discs' where generalized circles are Möbius transforms of the unitary group $\UU(r_{n+1})$ and 
generalized 'discs' are such images of $\DD(r_{n+1})$. Here, we will give some direct calculations. Let $\Im(\delta^z_{0,n})=I>\nul$. First, we show the following. 
$$\text{Claim 1:} \quad \hat \Ab^z_n=\left\{\tfrac12 \imath I^{-1}+\tfrac12 \sqrt{I^{-1}} R \sqrt{I^{-1}}\,:\,R^*R\,<\,1 \,\right\}$$
Clearly, by redefining $B$, the set $\hat A^z_n$ is given by $\{(\sqrt{I}B\sqrt{I}-\imath I)^{-1}\,:\,\Im(B)<0\}$.
Now, equating this expression in $B$ with the one in $R$ in the claim gives the relation
$$
R\,=\,2(B-\imath \II)^{-1}\,-\,\imath\,\II  
$$
For $\Im(B)<0$ such $R$ is well defined. Then,
$$
R^*R\,=\,4(B^*+\imath\II)^{-1}(B-\imath\II)^{-1} +2\imath (B-\imath \II)^{-1} - 2 \imath (B^*+\imath \II)^{-1}+\II
$$
Therefore, $R^*R<\II$ becomes equivalent to
$$
\nul \,>\,4\cdot \II\,+\,2\imath(B^*+\imath \II)\,-\, 2 \imath (B-\imath \II)\,=\, 2\imath (B^* - B) \,=\, 4 \Im(B)
$$
which is indeed fulfilled.\\
On the other hand, having $R^*R<\II$ we can define
$B\,=\,2(R+\imath \II)^{-1}\,+\,\imath\II$ and the calculation above shows $\Im(B)<0$ for $R^*R<\II$. This gives claim 1.  \hfill $\blacksquare$

The above relations also directly give a one to one correspondence  between $\{B^ *=B\}$ and $\{R^*R=\II\,:\,-\imath \not \in \spec(R)\}$ as well as between $\{B^*B\leq \II\}$
and $\{R^*R\leq \II\,:\,-\imath \not \in \spec(R)\}$. Thus, we obtain directly that
\begin{align*}
\breve \Ab^z_n\,&\,=\left\{\tfrac12 \imath I^{-1}+\tfrac12 \sqrt{I^{-1}} R \sqrt{I^{-1}}\,:\,R\in \UU(r_{n+1}) \,\right\} \qtx{and} \\
\Ab^z_n\,&=\left\{ \tfrac12 \imath I^{-1}+\tfrac12 \sqrt{I^{-1}} R \sqrt{I^{-1}}\,:\,R\in \DD(r_{n+1})\right\}\;.
\end{align*}
$\breve \WW^z_n$ and $\WW^z_n$ are of the form $\{ \alpha+\beta R \gamma \}$ with $\alpha\in \CC$, $\beta$ being a complex row vector,  $\gamma$ a complex column vector and $R$ varying in $\UU(r_{n+1})$ or $\DD(r_{n+1})$, respectively. 
For $r_{n+1}=1$ we see immediately that $\Ab^z_n$ is a disc in $\CC$ and $\breve \Ab^z_n$ the boundary circle and the same is true for $\breve \WW^z_n$ and $\WW^z_n$. For $r_{n+1}\geq 2$ both represent the same disc:
$$
\text{Claim 2: Let $r_{n+1}\geq 2$, then:}\quad  
\breve \WW^z_n= \WW^z_n
$$
Changing $R$ with $URV$ where $U, V$ are adequate unitaries, we can assume that $\beta^*$ and $\gamma$ are multiples of $e_1\in\RR^{r_{n+1}}$, the first canonical basis vector in $\RR^{r_n}$. This means $\breve \WW^z_n$ and $\WW^z_n$ are given by ${\alpha+c e_1^* R e_1}$ where $\alpha, c\in\CC$, $e_1^* R e_1$ is the top left entry of the matrix $R$ and $R$ varies in $\UU(r_{n+1})$ or $\DD(r_{n+1})$, respectively.
It is a well known fact of linear algebra that for $r_{n+1}\geq 2$ 
we find
\begin{align*}
\{e_1^* R e_1\,:\,R\in \UU(r_{n+1})\}\,&=\,
\{z\in\CC\,:\,|z|\leq 1\}\,=\,\{e_1^* R e_1\,:\,R\in \DD(r_{n+1})\}\,.
\end{align*}
This shows that $\WW^z_n$ is indeed a closed disc and $\breve\WW^z_n=\WW^z_n$.
\hfill $\blacksquare$\\[.3cm]
For part (ii) note that from the formulas above the representation $\alpha+ce_1^* R e_1$ gives the radius
\begin{equation}\label{eq-r_n}
r_{z,n}=|c|=\tfrac12 \|\beta^z_{0,n} \sqrt{I^{-1}} \|\, \| \sqrt{I^{-1}} \gamma^z_{0,n}\|\,. 
\end{equation}
We first prove:
$$
\text{Claim 3: $\im(z)>0, \,z\not\in \bigcup_{k=0}^\infty$ then:}\quad
\| \beta^z_{0,n} \sqrt{I^{-1}}\|^2\,=\,\frac{1}
{\im(z)\,\min_{\psi\in\Psi^z} \left(\sum_{k=0}^n \|\psi_k\|^2\, \right)}\,.
$$
Let us use the notation $\psi_{0,n}$ for $\bigoplus_{k=0}^n \psi_k \in \CC^{\GG_{0,n}}$.  Then from $H\psi=z\psi$ we have
$$
(H\psi)_{0,n} = H_{0,n} \psi_{0,n} - \hat\Phi_{n} \Upsilon_{n+1}^* \psi_{n+1} = z \psi_{0,n}
$$
Using $B=\Upsilon_{n+1}^* \psi_{n+1}$ this gives $\psi_{0,n}=(H_{0,n}-z)^{-1}\hat\Phi_{n} B$ and thus
$$
\beta^z_{0,n} B\,=\,\hat \Upsilon_{0}^* (H_{0,n}-z)^{-1} \hat\Phi_{n} B\,=\,\hat\Upsilon_{0}^*\psi_{0,n}=\Upsilon_0^*\psi_0=1
$$
as well as
$$
\im(z)\|\psi_{0,n}\|^2\,=\,\im(z)\,B^* \hat\Phi_{n}^* |H_{0,n}-z|^{-2} \hat\Phi_{n} B\,=\,B^* I B\,=\, \|\sqrt{I} B\|^2
$$
Now, $\beta^z_{0,n} \sqrt{I^{-1}} \sqrt{I} B =1$ and, therefore, using Cauchy-Schwarz we get
$$
\im(z) \|\psi_{0,n}\|^2 \|\beta^z_{0,n} \sqrt{I^{-1}}\|^2\,=\,
\|\sqrt{I} B\|^2\, \|\beta^z_{0,n} \sqrt{I^{-1}}\|^2\,\geq\,1
$$
Now, using Proposition~\ref{prop-3} one finds $\psi\in\Psi^z$
such that 
$$ B\,=\,\Upsilon_{n+1}^*\psi_{n+1}\,=\,
I^{-1} (\beta^z_{0,n})^* \,/\, (\beta^z_{0,n} I^{-1} (\beta^z_{0,n})^*)$$
 which gives equality above and proves Claim 3. Note that for  $z\not\in \bigcup_{k=0}^\infty A_k$ the row vector $\beta^z_{0,n}$ is not zero and as $I^{-1}>0$ we are not dividing by zero. 
Realizing that $\gamma^z_{0,n}=(\beta^{\bar{z}}_{0,n})^*$ one finds similarly
$$
\| \sqrt{I^{-1}}\gamma^z_{0,n}\|^2 \,=\,\frac{1}
{\im(z)\,\min_{\psi\in\Psi^{\bar z}} \left(\sum_{k=0}^n \|\psi_k\|^2\, \right)}\,.
$$
Together with Claim~3 and \eqref{eq-r_n} this proves part (ii).
\hfill $\blacksquare$\\[.3cm]
For part (iii) note that by Proposition~\ref{prop-bd} and \eqref{eq-sg-1} we find
$$
\alpha^z_{0,n+1}\,=\,\alpha_{0,n}^z\,+\,\beta^z_{0,n}((\alpha^z_{n+1})^{-1}-\delta^z_{0,n} )^{-1}\gamma^z_{0,n}\,=\,\alpha^{\alpha^z_{n+1},z}_{0,n}\,\in\, \WW^z_n\,.
$$
Formally varying the matrix-potential $V_{n+1}$ (which changes $\alpha_{n+1}^z$) then shows
$\breve \WW^z_{n+1} \subset  \WW^z_n$ which by (i) implies $\WW^z_{n+1} \subset \WW^z_n$.
\end{proof}

Now, as we see from (iii), the radius $r_{z,n}$ of the Weyl disc is shrinking and we can define
$r_{z,\infty}=\lim_{n\to\infty} r_{z,n}$ which is the radius of the limiting disc $\bigcap_{n=0}^\infty \WW^z_n$.
If $r_{z,\infty}=0$ then the intersection consists of one point and we have a limit point. If $r_{z,\infty}>0$ then we have a limit disc.

\begin{prop} \label{prop-W2}
{\rm (i)} Let $z, \bar z \not \in \bigcup_{n=0}^\infty A_n$. 
We have a limit disc, $r_{z,\infty}>0$, if and only if there exist  $\psi^z \in \Psi^z \cap \ell^2(\GG)$ and $\psi^{\bar z}\in \Psi^{\bar z} \cap \ell^2(\GG)$. In particular, in this case both deficiency indices are at least one, $d^\pm=\dim \ker(H_{min}^*\pm \imath)\geq 1$.\\
{\rm (ii)}
 If $H$ is self-adjoint on its natural domain, then we have the limit point case and 
for any sequence $w^z_n\in \WW^z_n$ we find
$$
\lim_{n\to\infty} w^z_n\,=\,\langle P_0 \Upsilon_0\;|\; (H-z)^{-1} P_0 \Upsilon_0\rangle\,=\,
\int (\lambda-z)^{-1}\,\mu_0(d\lambda)\;,
$$ and thus,
$$
\bigcap_{n=1}^\infty \WW^z_n=\{\,\langle P_0 \Upsilon_0\;|\; (H-z)^{-1} P_0 \Upsilon_0\rangle\,\}\quad \text{for any $z\in \CC^+$.} 
$$
\end{prop}

\begin{proof}
If $\psi^z\in \Psi^z\cap \ell^2(\GG)$ and $\psi^{\bar z} \in \Psi^{\bar z}\cap \ell^2(\GG)$ exist then Proposition~\ref{prop-W1}~(ii) immediately gives $r_{z,\infty}>0$.\\
Now let $r_{z,\infty}>0$. Then by 
 Proposition~\ref{prop-W1}~(ii)
you find $\psi^{z,n} \in \Psi^z$ such that
$$
\sum_{k=0}^n \|\psi^{z,n}_k\|^2\,<C\,
$$
for $n\in \NN$. Now, take $ \phi^{z,n}_k= \psi^{z,n}$ for $k\leq n$ and $\phi^{z,n}_k=0_k\in \CC^{s_k}$ for $k>n$.
This way, $\| \phi^{z,n}\|^2 < C$ for all $n$.
Then, there is a weakly convergent sub-sequence 
$\phi^{z,n_j} \to \psi^z \in \ell^2(\GG)$. In particular, $\psi^{z,n_j}_k \to \psi^z_k$ for any $k\in \NN_0$. Using continuity in the eigenvalue equation we find $\psi^z \in \Psi^z$.
Note, $H_{min}^*\psi^z= H\psi^z=z \psi^z$ with $\psi^z \neq 0,\,\psi^z \in \ell^2(\GG)$, and hence $d^+\geq 1$.
The same arguments hold for $z$ replaced by $\bar z$.\\[.2cm]
For part (ii), if $H$ is self-adjoint with its natural domain, i.e. the compactly supported vectors $\HH_{c}$ form a core, then $H_{0,n}+\hat\Phi_n A_n \hat\Phi_n^*$ (which act as zero in $\ell^2(\GG_{0,n})^\perp \subset \ell^2(\GG)$) converge to $H$ in strong resolvent sense\footnote{Indeed, for $\psi\in \HH_{c}$ it is obvious that $(H_{0,n}+\hat\Phi_n S_n \hat \Phi_n^*)\psi \to H \psi$ and if $\HH_{c}$ is a core, then this implies strong resolvent convergence} for any choice of $A_n=A_n^*$ as $n\to \infty$. This immediately implies the result.
\end{proof}

Note that Proposition~\ref{prop-W1} and \ref{prop-W2} prove Theorem~\ref{Theo-2}.

\section{Spectral averaging formula}

We let 
$$
\nu_C(d\lambda)\,=\,\frac{1}{\pi}\,\frac{1}{1+\lambda^2}\,d\lambda
$$
denote the standard Cauchy distribution on $\RR$.
The Cauchy distribution can be obtained as the pull back measure of the standard Haar measure on the unit circle $S^1\subset \CC$ 
by the Caley transform,
$$
\lambda\,\mapsto \,\Cc\cdot \lambda\,:=\,\pmat{1 & -\imath  \\ 1 & \imath }\,\cdot\,\lambda\,=\,(\lambda-\imath )(\lambda+\imath)^{-1}\;
$$
as a map from $\RR$ to $S^1=\{z\in\CC\,:\,|z|=1\}$.
Then we define the $n$-th level averaged spectral measure at $\Upsilon_0$ by
$$
 \bar\mu_{n}(f)\,:=\, \int_{\RR^{r_{n+1}}}\langle \hat \Upsilon_0\,|\,f(H_n-\hat\Phi_n\, \diag(a_1,\ldots,a_{r_{n+1}})\, \hat\Phi_n^*)\,|\,\hat\Upsilon_0\rangle\,
 \prod_{j=1}^{r_{n+1}} \nu_C(da_j)\,.
$$
that is we consider real diagonal matrices $A$ only and integrate the spectral measures over the product measure of the Cauchy distribution on the diagonal entries.

If $f\,:\,z \mapsto f(z)$ is a complex analytic function on 
the upper half plane $\HH^+=\{z\in \C\,:\,\im(z)>0\}$ and continuous up to the boundary and at infinity, then Cauchy's formula gives
$$
\int_{\RR} f(\lambda)\,\nu^C(d\lambda)\,=\,\int_{0}^{1} f(\Cc^{-1}\cdot e^{2\pi \imath t})\,dt\,=\,f(\Cc^{-1} \cdot 0)\,=\,
f(\imath)\,.
$$
 Note that 
$$
A\,\mapsto\, \alpha_{0,n}^{z,A}\,=\,\alpha_{0,n}^z+\beta_{0,n}^z\left[ \pmat{\II & \nul \\ -\delta_{0,n}^z& \II}
\Cc^{-1}\,\cdot\,\left( \Cc \cdot A\right) \right]
$$
is of this form in each diagonal coordinate of $A=\diag(a_1,\ldots,a_m)$. Therefore, the Stieltjes transform of $\bar \mu_n$ is given by
\begin{align}
\bar S_n(z)\,:=\,\int (z-\lambda)^{-1} \bar\mu(d\lambda)\,&=\,
\alpha^{z,\imath \II}_{0,n}\,.
\end{align}
The limit point property in the case where $H$ is uniquely self-adjoint gives
$\alpha^{z,\imath\II}_{0,n} \to \langle P_0 \Upsilon_0\,,\,(H-z)^{-1} P_0 \Upsilon_0\rangle$
which reflects the fact that $\bar\mu_n\to \mu$ weakly for $n\to\infty$.
 
\begin{proof}[Proof of Theorem~\ref{Theo-1}]
The key for proving Theorem~\ref{Theo-1} is to obtain the density of the spectral averaged measures $\bar \mu_n$ and identifying its singular part.
For this we consider the Stieltjes transform 
$$
z \,\mapsto\,  \bar S_n(z)=\alpha^{z,\imath \II}_{0,n}\,=\,\alpha^z_{0,n}\,+\,\imath \beta_{0,n}^z (\II-\imath \delta^z_{0,n})^{-1} \gamma^z_{0,n}  \qquad \text{for $\im(z)>0$}.
$$
For $\lambda\not \in A^1_{0,n} \subset \RR$ the limit $\bar S_n(\lambda+\imath \eta)$ for $\eta \downarrow 0$ exists. Thus, the singular part of $\bar \mu_n$ is supported on the finite set $A^1_{0,n}$
and given by a point measure $\nu_n$. The absolutely continuous part of $\bar \mu_n$ has the density (Radon-Nikodym derivative with respect to the Lebesgue measure $d\lambda$) 
\begin{align*}
\pi\,\frac{d \bar \mu_n}{d\lambda}\,&=\, \im(\bar S_n(\lambda+\imath0)\,=\, (\gamma_{0,n}^\lambda)^*\, \Re\left(\II-\imath \delta^\lambda_{0,n}\right)^{-1} \gamma^\lambda_{0,n} \\
&=\, (\gamma^\lambda_{0,n})^*\,\left(\II+(\delta^\lambda_{0,n})^2\right)^{-1}\,\gamma^\lambda_{0,n}\,=\, \|\bfv^\lambda_n\|^2 \qtx{for} \lambda\in \RR \setminus A^1_{0,n}
\end{align*}
where $\Re( A)= (A+A^*)/2$ denotes the real part of a matrix in $C^*$ algebra sense and we define
$$
\bfv^\lambda_n\,:=\, \left(1+(\delta^\lambda_{0,n})^2\right)^{-1/2} \gamma^\lambda_{0,n}\,.
$$
We used the facts that $\beta^\lambda_{0,n}=(\gamma^\lambda_{0,n})^*$,  $\alpha^\lambda_{0,n}$ is a real number and $\delta^\lambda_{0,n}$ is Hermitian (note $\lambda\in\RR$).
Also note that $\gamma^\lambda_{0,n}$ is indeed a $r_n \times 1$ matrix and hence a vector in $\CC^{r_n}$.
Now assume $\lambda \not \in A_{0,n}$, let $B\in \BB^\lambda_{0,n}$ and consider
$$
\bfu^\lambda_n \,:=\,\pmat{B \\ \delta^\lambda_{0,n} B}\, \in\, \D^\lambda_{n} \qtx{and}
\bfw^\lambda_n\,:=\,\left(\II+(\delta^\lambda_{0,n})^2 \right)^{1/2} B\;.
$$
Then, $\|\bfu^\lambda_n\|^2=\|\bfw^\lambda_n\|^2$ and  $1=\beta^\lambda_{0,n} B\,=\,(\bfv^\lambda_n)^*\bfw^\lambda_n=(\bfv^\lambda_n,\bfw^\lambda_n)$. Thus, the Cauchy-Schwartz inequality gives us
$$
\| \bfu^\lambda_n\|^{-2}\,=\, \|\bfw^\lambda_n\|^{-2}\,\leq\, \|\bfv^\lambda_n\|^2\,=\, \im \bar S_n(\lambda+\imath 0)\,.
$$
Using $ B_0=\|\bfv^\lambda_n\|^{-2} \left(\II+(\delta^\lambda_{0,n})^2\right)\gamma^\lambda_{0,n} \,\in\,\BB^\lambda_{0,n}$ and $\hat \bfu^\lambda_n=\smat{B_0 \\ \delta^\lambda_{0,n} B_0} \in \D^\lambda_{n}$
we get equality, therefore
$$
\im \bar S_n(\lambda+\imath 0)\,=\,\|\bfv^\lambda_n\|^2\,=\,\max_{\bfu^\lambda_n \in \D^\lambda_n}\,\|\bfu^\lambda_n\|^{-2}\,=\, \left( \min_{\bfu^\lambda_n \in \D^\lambda_n} \|\bfu^\lambda_n\|^2\right)^{-1} \;.
$$
Thus,
\begin{equation} \label{eq-mu-bar}
\bar \mu_n(d\lambda)\,-\, \nu_n(d\lambda)\,=\, \frac{1}{\pi} \|\bfv^\lambda_n \|^2 d\lambda\,=\, \frac{d\lambda}{\pi\,\min_{\bfu^\lambda_n \in \D^\lambda_n} \|\bfu^\lambda_n\|^2} \;.
\end{equation}
Note that this is a generalization of the spectral averaging formula \cite[Theorem III.3.2]{CL}.

Now let us consider the singular part $\nu_n$ of $\bar \mu_n$ closer. We already know that it is supported on $A^1_{0,n}$ which is finite and
hence it is a point measure. As the Cauchy distribution is absolutely continuous we see that $\nu_n(\{\lambda_0\})=\bar\mu_n(\{\lambda_0\})>0$ 
for $\lambda_0\in A^1_{0,n}$ if and only if $\Aa=\{\bfa\in \RR^{r_{n+1}}\,:\, \det(H_{0,n}-\hat \Phi_n \diag(\bfa)\hat\Phi_n^*-\lambda_0\II)= 0 \}$ has positive (Lebesgue) measure.
It is a well known linear algebra fact (using Fubini and rank-one perturbations) that this can only occur if there is an eigenvector $\psi\neq 0,\,\psi\in \ell^2(\GG_{0,n})$ which is orthogonal to all column vectors of $\Phi_n$ and has some overlap with $\Upsilon_0$. This means
that $H_{0,n} \psi=\lambda_0  \psi$, $\Phi_n^* \psi_n = \nul$ and $\langle \upsilon_0, \psi_0\rangle=\Upsilon_0^* \psi_0\neq 0$. 
This in fact implies that $\psi$ is an eigenvector of $H_{0,n}-\hat\Phi_n A \hat \Phi_n^*$ for any $A\in \Her(r_{n+1})$. 
Moreover, only such eigenvectors $\psi$ contribute to $\bar \mu_n(\{\lambda_0\})$. Hence, let 
$$
\HH_{n,\lambda_0}:=\ker(H_{n,0}-\lambda_0 \II) \cap {\rm Ran}(\hat\Phi_n)^\perp
$$ 
be the part of the eigenspace of $H_{0,n}$ for $\lambda_0$
which is orthogonal to the range\footnote{we interpret $\hat\Phi_n$ as a map from $\CC^{r_{n+1}}$ to $\ell^2(\GG_{0,n})$} of $\hat\Phi_n$ (orthogonal to all the column vectors of $\hat\Phi_n$).
Let 
$$
\PP_{n,\lambda_0}\,:\,\CC^{r_{n+1}}\to \HH_{n\,\lambda_0} ,\quad \text{be the orthogonal projection, then}
$$
$$
\nu_n(\{\lambda_0\})\,=\,\bar \mu_n(\{\lambda_0\})\,=\, \left\| \PP_{n,\lambda_0} \hat\Upsilon_0 \right\|^2\;.
$$
From $\Phi_n^* \psi_n=\nul$ it also follows that any eigenvector $\psi \in \HH_{n,\lambda_0}$ is an eigenvector  of $H_{0,m}-\hat \Phi_m A \hat \Phi_m^*$ for any $A\in \Her(r_{m+1})$ and any $m\geq n$,
that is to say $\HH_{n,\lambda_0} \subset \HH_{m,\lambda_0}$ for any $m\geq n$.
Here, formally we have to consider the natural embedding of $\HH_{n,\lambda_0} \subset \ell^2(\GG_{0,n})$ into $\ell^2(\GG_{0,m})$.
This immediately implies for $m\geq n$ that $\nu_m(\{\lambda_0\})\geq \nu_n(\{\lambda_0\})$ and $\lambda_0 \in \bigcap_{m\geq n} A^1_{0,m}$.
Hence, $\nu_n$ is a sequence of growing positive point measures which is bounded by $\mu$ and thus it converges to a point measure $\nu$ supported on
$\bigcup_n \bigcap_{m\geq n} A^1_{0,m}$. 
Moreover, any eigenvector $\psi_0 \in \HH_{n,\lambda_0}$ is an eigenvector of $H$ (or better the formal embedding of $\psi_0$ into $\ell^2(\GG)$ is) and $\nu$ is generated by those eigenvectors.
This means, let
$$
\HH_{\lambda_0}\,=\,\overline{\bigcup_{n=1}^\infty P_{0,n} \HH_{n,\lambda_0}\,}\,\subset\,\ker(H-\lambda_0)\;.
$$
be the closure of the union of the embeddings of $\HH_{n,\lambda_0}$ into $\ell^2(\GG)$ and let
$\PP_{\lambda_0} \,:\,\ell^2(\GG) \to \HH_{\lambda_0}$ be the orthogonal projection. Then, $\nu(\{\lambda_0\})=\|\PP_{\lambda_0} \hat\Upsilon_0\|^2$ where we now consider $\hat \Upsilon_0$ as an element in $\ell^2(\GG)$ (that is $\hat \Upsilon_0=P_0\Upsilon_0$, here).
The convergence of $\bar \mu_n \to \mu$ gives
$$
\mu(d\lambda)\,-\,\nu(d\lambda)\,=\,\lim_{n\to\infty} \frac{\|\bfv^\lambda_n\|^2}{\pi}\,=\,\lim_{n\to\infty} \,\frac{d\lambda}{\pi\,\min_{\bfu^\lambda_n \in \D^\lambda_n} \|\bfu^\lambda_n\|^2}
$$
where the limit has to be understood as a weak limit of finite measures.
This proves the theorem.
\end{proof}

We like to remark that the measures $\nu_n$ and hence $\nu$ need not be zero. Already in the one-channel case there can be compactly supported eigenfunctions, cf. \cite{Sa-AT, Sa-OC}.

\section{Criteria for absolutely continuous spectrum}

In this section we will prove Theorem~\ref{Theo-1a} and Theorem~\ref{Theo-4}.
First of all, we note that for any choice of matrices $\lambda\mapsto T^\lambda_{0,n} \in \TT^\lambda_{0,n}$ we find
$$
\|\bfv^\lambda_n\|^2 \,\geq\, \|T^\lambda_{0,n} \smat{1\\0}\|^{-2}\;
$$
with $\bfv^\lambda_n\,:=\,\left(1+(\delta^\lambda_{0,n})^2\right)^{-1/2}\,\gamma^\lambda_{0,n}$ as defined above.
As this density appears through an average of spectral measures at $\Upsilon_0$ over a probability distribution, we find
$$
\int \|\bfv^\lambda_n\|^2\,d\lambda \,\leq\, \|\Upsilon_0\|^2
$$
and, hence, $\lambda \mapsto  \|T^\lambda_{0,n} \smat{1\\0}\|^{-2}$ is a positive $L^1$ function for any measurable choice of $\lambda \mapsto T^\lambda_{0,n}$. Moreover, any infimum over such a function is also in $L^1$.

\begin{proof}[Proof of Theorem~\ref{Theo-1a}]
Part (i). We have
$$
\mu(d\lambda)\,\geq\,
(\mu-\nu)(d\lambda)\,=\,
\lim_{k\to \infty} \frac1\pi \|\bfv^\lambda_{n_k}\|^2\,d\lambda\,\geq\, \left( \sup_{k} \left\|T^\lambda_{0,n_k} \pmat{1\\0}\,\right\|^2 \right)^{-1}\,d\lambda\;.
$$
On the right hand side we have an absolutely continuous measure with a positive density on $(a,b)$ by the assumptions of Theorem~\ref{Theo-1a}.
Therefore, the bigger positive measure $\mu$ has to have an absolutely continuous part with a bigger support than $[a,b]$.
This means, $[a,b]\subset \spec_{ac}(H)$.\\[.3cm]
To show part (ii)  note that using $\| \bfv^\lambda_n\|^{-2} \leq \|T^\lambda_{0,n} \smat{1\\0}\|$ 
the condition in the Theorem implies
$$
\liminf_{n\to\infty} \int_K -\log\left(\|\bfv^\lambda_n\|^2 \frac{1}{w(\lambda)}\right)\,w(\lambda)\,d\lambda\,<\,\infty\;
$$
for any compact set $K\subset (a,b)$ of positive Lebesgue measure.
Using the arguments as in \cite{DK}, or Proposition~\ref{prop-B1}, it follows that the limit measure $\mu-\nu$ has an absolutely continuous part everywhere in $(a,b)$.
\end{proof}

\begin{proof}[Proof of Theorem~\ref{Theo-4}]
By the arguments in the previous section we find a matrix $\hat T^\lambda_{0,n}\in \TT^\lambda_{0,n}$ such that
$$
\|\bfv^\lambda_{n}\|\,=\,\left\|\hat T^\lambda_{0,n}\pmat{1\\0}\right\|^{-1}\,.
$$
From Proposition~\ref{prop-2} we find
$$
1\,=\,\pmat{\bar u_n(\lambda)&1} J_1 \pmat{1\\0}\,=\,  \left(T^\lambda_{0,n} \pmat{u_n(\lambda) \\ 1}\right)^* J_{r_{n+1}} \hat T^\lambda_{0,n} \pmat{1\\0}
$$
and for that reason
$$
1\,\leq\, \left\| T^\lambda_{0,n} \pmat{u_n(\lambda)\\ 1}\right\|\,\|J_{r_{n+1}}\|\,\left\|\hat T^\lambda_{0,n} \pmat{1\\0}\right\|\,=\,
\left\| T^\lambda_{0,n} \pmat{u_n(\lambda)\\ 1}\right\|\,/\, \|\bfv^\lambda_n\|\;.
$$
It follows immediately for $f_n(\lambda):=\frac1\pi \|\bfv^\lambda_n\|^2$ that 
$$
\liminf_{n\to\infty} \int_a^b\,|f_n(\lambda)|^{p}\,d\lambda \,\leq\,\liminf_{n\to\infty} \int_a^b \left\|\frac1\pi T^\lambda_{0,n}\pmat{u_n(\lambda)\\1} \right\|^{2p} \,d\lambda\,<\,\infty\,.
$$
Hence, there is a sub-sequence $f_{n_k}$ such that the $L^p$ norm of $f_{n_k}$ in $(a,b)$ is uniformly bounded along this sub-sequence.
Using separability, reflexiveness and weak-$*$ compactness, we find a further sub-sequence which converges weakly in $L^p$, i.e.
$f_{n_j} \rightharpoonup f$. Then, integrating against bounded continuous functions $g$  in $(a,b)$ which are also in $L^q(a,b)$ where $1/q+1/p=1$, we see that
$$
(\mu-\nu)(g)\,=\,\lim_{k\to \infty} \int_a^b g(\lambda) f_{n_k}(\lambda)\,d\lambda\,=\,
\int_a^b g(\lambda)\, f(\lambda)\,d\lambda
$$
 and hence $(\mu-\nu)(d\lambda)= f(\lambda)d\lambda$ in the interval $(a,b)$.
 This proves part (i). Part (ii) follows immediately from part (i) and Theorem~\ref{Theo-1a}~(ii).
\end{proof}

\section{Application to random models}

In this section we first prove Theorem~\ref{Theo-5} and then Theorem~\ref{Theo-6} and Theorem~\ref{Theo-7} will follow with little extra work.
Thus, consider the operator $H^{(1)}_\omega=\Delta^{(1)}+V_\omega$.
\begin{proof}[Proof of Theorem~\ref{Theo-5}]
The matrix-potential part of $\Delta^{(1)}$ in the $n$-th shell is given by
$$
A_n\,=\,\diag(a_1, a_2,\ldots, a_{s_n})\,\in\,\RR^{s_n \times s_n}\,.
$$
We will use $\Phi_n=\II_{s_n}$ for $n\geq 0$, giving $\Upsilon_n=\pmat{\II_{s_{n-1}} \\ \nul}\in\RR^{s_n \times s_{n-1}}$ 
for $n\geq 1$. Furthermore we let $\Upsilon_0\in \ell^2(S_0)$ be some normalized root vector, so that generally $\Upsilon_n^* \Upsilon_n=\one$.
Similar to Remark~\ref{rem-T}~(ii), the special choices $B=(A+V_n-z)\Upsilon_n$ and $\bbb=\Upsilon_n$ lead to the transfer matrices
$$
T^z_n\,=\,\pmat{(A_n+V_n-z\II)\Upsilon_n & -\Upsilon_n \\ \Upsilon_n & \nul}\,=\,\pmat{A_n+V_n-z\II & -\II \\ \II & \nul} \pmat{\Upsilon_n & \nul \\ \nul & \Upsilon_n}\,.
$$
These choices lead to nice entire functions in $z\mapsto T^z_n$.

Let us also note that if $H^{(1)}_\omega\psi=\lambda\psi$ and $\psi_m=\nul=\psi_{m+1}$, then, it follows inductively that $\psi_n=\nul$ for all $n\leq m$ (as the width is growing, $s_{n+1}\geq s_n$). Hence, there is no non-zero compactly supported eigenvector. Thus, the measure $\nu$ as in Theorem~\ref{Theo-1} is  zero.\\[.2cm] 
\noindent {\bf Claim:} We have
$$
\liminf_{n \to \infty} \int_a^b \EE \|T^\lambda_{0,n}\|^4\,d\lambda\,<\,\infty\;
$$ 
for any compact sub-interval $[a,b]$ of $I_1=(-2+\sup_j a_j,2+\inf_j a_j) $. \\
Using Theorem~\ref{Theo-4}, the argument above, Fatou's lemma and Fubini,
it will follow\footnote{Indeed, this implies $
\EE\,\liminf_{n\to \infty} \int_a^b \|T^\lambda_{0,n}\|^4\,d\lambda\,<\,\infty\;
$ by Fubini and  Fatou's lemma and a bound in expectation implies almost sure boundedness to use Theorem~\ref{Theo-4}~(ii).} 
that the spectral measure at $P_0\Upsilon_0$ is purely absolutely continuous in $I_1$.\\[.2cm]
So let us prove the claim.
By sub-multiplicativity of norms, we can exchange $T^\lambda_{0,n}$ and $T^\lambda_n$ by $Q_n(\lambda)^{-1}T^\lambda_{1,n} Q_{0}(\lambda)$ and $Q_n(\lambda)^{-1}T^\lambda_n Q_{n-1}(\lambda)$, respectively, 
where $Q_n(\lambda)\in \GL(2s_n)$ (with $s_{-1}=1$)
and $\|Q_n(\lambda)\|$, $\|Q_n(\lambda)^{-1}\|$ are uniformly bounded for $\lambda\in[a,b]$ and $n\in \ZZ_+$.
Using the bound
$\|T\|\leq \|Te_1\|+\|Te_2\|$ for any orthonormal basis $(e_1, e_2)$ in $\CC^2$ and any $s\times 2$ matrix $T$, it is, thus, sufficient to prove 
\begin{equation}\label{eq-need1}
\liminf_{n \to \infty} \int_a^b \EE \|Q_n(\lambda)^{-1}\,T^\lambda_{0,n}\,Q_{-1}(\lambda)\, \bfu\|^4\,d\lambda\,<\,\infty\;
\end{equation}
for $Q_n(\lambda)\in \GL(2s_n)$ as above and all fixed $\bfu\in \CC^{2}$.

For $\lambda \in I_1$ all values on the diagonal of the diagonal matrix $A_n-\lambda \II$ are elements of $(-2,2)$ and we may write
$A_n-\lambda\II=2\cos(K_n)$ for some real diagonal matrix $K_n=K_n(\lambda)=\diag(k_1,\ldots,k_{s_n})$ with $k_j(\lambda)\in(0,\pi)$
and $2\cos(k_j)=a_j-\lambda$.
For the unperturbed  operator $\Delta^{(2)}$ (all $V_n=0$) the transfer matrices for $n\geq 1$ are equal to
$$
\widehat T^{\lambda}_n \,=\,\pmat{(A_n-\lambda\II)\Upsilon_n & -\Upsilon_n \\ \Upsilon_n & \nul}\,=\,\pmat{2\cos(K_n) & -\II \\ \II & \nul}
\pmat{\Upsilon_n \\ & \Upsilon_n}
$$
and we define
$$
Q_n\,=\,Q_n(\lambda)\,:=\,\pmat{e^{\imath K_n(\lambda)} & e^{-\imath K_n(\lambda)} \\ \II & \II}\;.
$$
Using $\cos(K_n) \Upsilon_n= \Upsilon_n \cos(K_{n-1})$ one obtains
$$
Q_n(\lambda)^{-1} \widehat T^{\lambda}_n Q_{n-1}(\lambda)\,=\,\pmat{e^{\imath K_n} \Upsilon_n &  \\ & e^{-\imath K_n} \Upsilon_n}\,=:\,\Rr_n
$$
where $\Rr_n=\Rr_n(\lambda)$ depends on $\lambda$ as well.
This implies
$$
Q_n^{-1} T^\lambda_n Q_{n-1}\,=\, \Rr_n\,+\, \Vv_n \qtx{where} \Vv_n=\Vv_n(\lambda)\,=\,Q_n(\lambda)^{-1} \pmat{V_n \Upsilon_n& \nul \\ \nul & \nul} Q_{n-1}(\lambda)\,.
$$
Note that $\Rr_n$ is an isometry, $\|\Rr_n u\|=\|u\|$, and for $\lambda\in[a,b]\subset I_1$, the values $e^{\imath k_j(\lambda)}$ are uniformly bounded away from $1$ and $-1$. Hence, 
$\|Q_n(\lambda)\|$ and $\|Q_n(\lambda)^{-1}\|$ are uniformly bounded in $[a,b]$ and $n$. Moreover, the random matrices $\Vv_n(\lambda)$ are independent and satisfy a bound as in \eqref{eq-V-estimate}  uniformly in $\lambda \in [a,b]$. 
We will now omit the dependence on $\lambda$ in most calculations.
Furthermore, we consider some starting vector $\bfu\in \CC^2$ and define
$$
\bfu_n\,:=\, Q_n^{-1} T^\lambda_{0,n} Q_{-1}\,\bfu \qtx{implying} \bfu_n\,=\, (\Rr_n+\Vv_n) \bfu_{n-1}\,.
$$
Then, we find
$$
\|\bfu_n\|^4 =(\bfu_n^* \bfu_n)^2=\left( \|\bfu_{n-1}\|^2+ \bfu_{n-1}^* \left( \Rr_n^* \Vv_n+\Vv_n^*\Rr_n\right) \bfu_{n-1}\,+\, \| \Vv_n \bfu_{n-1}\|^2\right)^2 
$$
Squaring out the last term, taking expectations and using the independence of $\Vv_n$ from $\bfu_{n-1}$ we find
$$
\EE(\|\bfu_n\|^4)\,\leq\, \EE(\|\bfu_{n-1}\|^4)\,\EE \left( 1+4\|\Vv_n\|^2+\|\Vv_n\|^4+4 \|\EE(\Vv_n)\|+2 \|\Vv_n\|^2+4\|\Vv_n\|^3 \right)\,.
$$
For the terms only having one $\Vv_n$ factor it is important to first do the expectation and afterwards the norm bound.
Using $\EE(\Vv_n\|^3)\leq \EE(\|\Vv_n\|^2+\|\Vv_n\|^4)$ and \eqref{eq-V-estimate} we see that the product over $n$ of the terms in the bracket on the right hand side remain bounded, uniformly for $\lambda \in [a,b]$. Hence, for $\bfu_n=\bfu_n(\lambda)$ we obtain
$\int_a^b \EE(\|\bfu_n(\lambda)\|^4 d\lambda$ is uniformly bounded in $n$ which implies the claim \eqref{eq-need1}. \hfill $\blacksquare$\\[.2cm]
As explained above, this applies for any $\Upsilon_0=\delta_{(0,j)}\in\ell^2(S_0)$. Thus, the spectral measure at $\delta_{(0,j)}$ is almost surely purely absolutely continuous in $I_1$, and there is spectrum.

Let us define the sub-graphs $\GG^+_{n}=\bigcup_{m\geq n} S_m$ and let $H^{(1)}_{\omega,n}$ be the restriction of $H^{(1)}_\omega$ to
$\GG^+_{n}$. Again, using Theorem~5 will give that for any point $(n,j)\in S_n$ the spectral measure at $\delta_{n,j}$ of the operator
$H^{(1)}_{\omega,n}$ is, almost surely, purely absolutely continuous in $I_1$.
As $\GG_1$ is countable, we find a set $\Omega_0\subset \Omega$ of probability one, $\PP(\Omega_0)=1$, such that for all $\omega\in \Omega_0$ and
all $(n,j)\in \GG_1$, the
spectral measure of $H^{(1)}_{\omega,n}$ at $\delta_{n,j}$ is purely absolutely continuous in $I_1$.
Combining this with a soft modification of \cite[Lemma~2.2.]{FLSSS} one obtains that the spectrum of $H^{(1)}_\omega$ is purely absolutely continuous in $I_1$ for all $\omega \in \Omega_0$.
\end{proof}

The proof of Theorem~\ref{Theo-6} follows directly from Theorem~\ref{Theo-5}.
\begin{proof}[Proof of Theorem~\ref{Theo-6}]
Recall for the binary rooted tree $\GG_2$ we have $\GG_2=\big\{(n,j)\,:\,n\in\NN_0, \,j\in\{1,\ldots,2^n\}\,\big\}$ and
$S_n=\{n\}\times \{1,\ldots, 2^n\}$. We define the normalized symmetric and anti-symmetric mean-field vectors $\chi_n,\,\zeta_n\,\in\,\CC^{2^n}$ 
$$
\chi_n\,:=\,\frac{1}{\sqrt{2^{n}}}\,\pmat{1 \\ \vdots \\ 1}\,\in\,\CC^{2^n}\qtx{and} \zeta_n\,:=\, \frac{1}{\sqrt{2}}\pmat{\chi_{n-1} \\ -\chi_{n-1}}\,\in\,\CC^{2^n},\,n\geq 1
$$
and we let $\nul_n=\chi_n-\chi_n$ denote the zero vector in $\CC^{2^n}$. Then, we define the matrices
$$
U_n\,:=\,\pmat{\chi_n & \zeta_n & \mat{\zeta_{n-1} \\ \nul_{n-1}} & \mat{\nul_{n-1} \\ \zeta_{n-1} } & \mat{\zeta_{n-2} \\ \nul_{n-2} \\ \nul_{n-2} \\ \nul_{n-2} }& \mat{\nul_{n-2} \\ \zeta_{n-2} \\ \nul_{n-2} \\ \nul_{n-2} }  &  
\cdots & \mat{\zeta_1 \\ \nul_1 \\ \vdots \\ \nul_1} & \cdots & \mat{\nul_1 \\ \vdots \\ \nul_1 \\ \zeta_1} }\,.
$$
It is not hard to check that $U_n$ is a unitary $2^n \times 2^n$ matrix and we can define the unitary operator $\Uu$ on
$\ell^2(\GG_2)=\bigoplus_n \ell^2(S_n)$ by
$$
(\Uu \psi)_n\,=\, U_n\,\psi_n\qtx{where} \psi_n \in \ell^2(S_n), \psi=\bigoplus_n \psi_n\,.
$$
Let $\Psi^{(n,k)}=\bigoplus_n \Psi^{(n,k)}_n$ for $k \leq 2^n$ be defined by 
$$
 \Psi^{(n,k)}\,=\, U_n e_k\qtx{and}
\Psi^{(n,k)}_m\,=\,0 \qtx{for} m\neq n \,.
$$
where $e_k$ denotes the $k$-th canonical basis vector in $\CC^{2^n}$. Hence,  $U_n e_k$ is simply th $k$-th column vector of the matrix $U_n$.
Then it is not hard to check that
$$
\Delta^{(2)} \Psi^{(n,k)} = -\sqrt{2} \Psi^{(n-1,k)} \,-\,\sqrt{2} \Psi^{(n+1,k)}
$$
where $\Psi^{(n-1,k)}=0$ for $k>2^{n-1}$ or $n=0$.
Thus, considering the model $\Delta^{(1)}$ on $\GG_1$ as above with the special case $s_n=2^n$ and $a_j=0$ for all $j\in \NN$ we find
$$
\Uu^* \,\Delta^{(2)}\,\Uu\,=\, \sqrt{2}\,\Delta^{(1)}\,,\quad
\Uu^*\,H^{(2)}_\omega\, \Uu\,=\,\sqrt{2}\,\Delta^{(1)}\,+\,\Uu^*\, V_\omega\, \Uu
$$
Now, $\Uu^* \,\V_\omega \,\Uu$ is a shell-matrix potential of the same type as $V_n$ where $V_n$ is replaced by $V'_n=U_n^* V_n U_n$.
In particular, it also satisfies \eqref{eq-V-estimate}. Therefore, Theorem~\ref{Theo-5} gives that the spectrum of $H^{(2)}_\omega$ is almost surely purely absolutely continuous in $(-2\sqrt{2}\,,\,2\sqrt{2})$, and there is spectrum.
\end{proof}

Now, Theorem~\ref{Theo-7}~(i) can also be regarded as a special case of Theorem~\ref{Theo-5}.
The step from Theorem~\ref{Theo-7}~(i) to Theorem~\ref{Theo-7}~(ii) will then  be done as in \cite{FHS3}.

\begin{proof}[Proof of Theorem~\ref{Theo-7}]
Without loss of generality we may assume that $A$ is diagonal. If not, let $A\in \Her(s)$ be diagonalized by the unitary $U$, 
that is $U^{*} A U = \diag(a_1,\ldots,a_s)$, 
and consider $\Uu^{*} \Delta^{(3)} \Uu$, $\Uu^* H^{(3)}_\omega \Uu$ with $(\Uu \psi)_n=U\psi_n$. There, $A$ and $V_n$ are exchanged by $U^* A U \in \Her(s)$ and $U^*V_n U$, respectively,
which is another sequence of random Hermitian matrices satisfying \eqref{eq-V-estimate}. 
Hence, we let 
\begin{equation}
A\,=\,\diag(a_1,\ldots, a_s)\;.
\end{equation}
With this basis change we also see that $\Delta^{(3)}$ is (unitarily equivalent to) a direct sum of one-dimensional operators $\bigoplus_{j=1}^s h_j$
where $(h_j u)_n=-u_{n-1}-u_{n+1}+a_j u_n$ is an operator on $\ell^2(\ZZ_+)$ given by the discrete Laplacian and a constant potential $a_j$. 
It is well known that the spectrum of $h_j$ is purely absolutely continuous and given by the interval $[a_j-2,a_j+2]$, so the spectrum of $\Delta^{(3)}$ is given by the union $I=\bigcup_{j=1}^s [a_j-2,a_j+2]$ of these bands. By \eqref{eq-V-estimate}, and the fact that $\#S_n=s$ is constant, the difference $H^{(3)}_\omega-\Delta^{(3)}$ is a Hilbert Schmidt-operator (almost surely). Thus, the essential spectrum of $H^{(3)}_\omega$ is also given by $I$, almost surely.
Moreover, $\Delta^{(3)}$ corresponds to $\Delta^{(1)}$ of Theorem~\ref{Theo-5} with fixed width $s_n=s$, $S_n=\{n\}\times S$ with $S=\{1,\ldots,s\}$ and
$I_3=\bigcap_j (a_j-2,a_j+2)$ corresponds to $I_1$ as above. Thus, Theorem~\ref{Theo-5} implies that, almost surely, 
the spectrum of $H^{(3)}_\omega$ is purely absolutely continuous in $I_3$.

For part (ii), first consider a subset $J\subset S=\{1,\ldots,s\}$ and define the canonical embedding $P_J\,:\,\ell^2(\ZZ_+\times J)\hookrightarrow \ell^2(\ZZ_+\times S)$. For $\lambda \in I_0$  the set $J(\lambda)=\{j\in S\,:\,|a_j-\lambda|<2\}$ is not empty. We also define
$\complement J:=S\setminus J$ and let 
$I_J=\{\lambda\,:\, J(\lambda)=J\} \setminus \bigcup_j \{a_j-2,a_j+2\}$.
Then, we may write 
$$
H^{(3)}_\omega\,\equiv\, \pmat{P_J^* H^{(3)}_\omega P_J & P_J^* H^{(3)}_\omega P_{\complement J} \\ P_{\complement J}^* H^{(3)}_\omega P_J & P_{\complement J}^* H^{(3)}_\omega P_{\complement J} }
$$
where $\equiv$ denotes unitary equivalence. By Theorem~\ref{Theo-7}~(i) $P_J^* H^{(3)}_\omega P_J$ has almost surely pure absolutely continuous spectrum in
$I_J\subset I$, and there is spectrum. By the considerations above on the essential spectrum of $H^{(3)}_\omega$, $P_{\complement J}^* H^{(3)}_\omega P_{\complement J}$ has no essential spectrum in $I_J$. 
As $P_J^* \Delta^{(3)} P_{\complement J}=\nul$, the off diagonal terms only come from the random potential $V_n$ and are almost surely Hilbert-Schmidt operators from $\ell^2(\ZZ_+\times J)$ to $\ell^2(\ZZ_+\times\complement J )$ and vice versa.
Thus, using \cite[Theorem~6]{FHS3} we find that $H$ has almost surely absolutely continuous spectrum in all of $I_J$ (not necessarily pure).
Noting that $\bigcup_{J\subset S} I_J$ is dense in $I$, part (ii) follows as
$I\subset \spec_{ac}(H^{(3)}_\omega) \subset \spec_{ess}(H^{(3)}_\omega)=I$ almost surely.
\end{proof}

\section*{Acknowledgment}

C.S. has received funding from the Chilean 
Fondo Nacional de Desarrollo Cient\'ifico y Tecnol\'gico (FONDECYT) 1161651 and the 
Iniciativa Cient\'ifica Milenio through the Nucleo Mileneo MESCD on 'Stochastic Models of Complex and Disordered Systems'.

\appendix

\section{Some linear algebra facts}

As before with $\KK^{l\times m}$ we define the set of $l \times m$ matrices over $\KK$ where $\KK=\RR$ or $\KK=\CC$.
For $A\in \CC^{l\times l}$ with $l\leq m$ of full rank $l$ there exist right inverses, that is
matrices $\hat A\in \CC^{m \times l}$ such that $A\hat A=\II_l$ where $\II_l$ is the $l\times l$ unit matrix. 

\begin{prop}\label{prop-A1}
Let $l\leq m \leq n$, $A\in \KK^{l \times m}$ of full rank $l$ and $B\in \KK^{m\times n}$ of full rank $m$ with $\KK=\RR$ or $\KK=\CC$.
Then $C=AB \in \KK^{l\times n}$ is of full rank $l$ and the set of right inverses $\hat C \in \KK^{n \times l}$  is given by the set of products $\hat B \hat A$ of right inverses to $B$ and $A$, i.e.
$$
\{\hat C\in \KK^{n \times l}\,:\,AB\hat C=\II_{l} \}\,=\,\{\hat B \hat A\,:\, \hat A\in \KK^{m\times l}, \hat B \in \KK^{n\times m},\;A\hat A=\II_{l}\,\wedge\,B\hat B=\II_{m}\,\}
$$
\end{prop}

\begin{proof}
We can regard $A$ as a surjective linear map from $\KK^m$ to $\KK^l$ and $B$ as a surjective linear map from $\KK^n$ to $\KK^l$. It is thus clear that $AB=A\circ B$ is a surjective linear map from
$\KK^n$ to $\KK^l$ and hence of full rank. We let $(e_k)_{k=1}^j$ denote the standard basis in $\KK^j$. Applying basis changes we can suppose that $\ker A$ is spanned by $(e_k)_{k>l}$ in $\KK^m$ and $\ker B$ is spanned by
$(e_k)_{k>m}$ in $\KK^n$. Clearly $\ker B \subset \ker(AB)$. A further basis change leaving $(e_k)_{k>m}$ invariant in $\KK^n$ we can further suppose that $\ker(AB)$ is spanned by $(e_k)_{k>l}$ in $\KK^n$.
The dividing the matrices into the block structure of sizes $l, m-l, n-m$ we find
$$
A=\pmat{A_0 &\nul}\;, \quad B=\pmat{B_{00} & B_{10} & \nul \\ B_{10} & B_{11} & \nul}\;,\quad
C:=AB=\pmat{C_{0} & \nul & \nul}
$$
where $A_0$ and $C_0$ are invertible $l\times l$ matrices.
Using $C=AB$ we find that $A_0 B_{01}=\nul$ and hence $B_{01}=A_0^{-1} A_0 B_{01}=\nul$.
Therefore, $B_{00}$ has to be an invertible $l\times l$ matrix and $B_{11}$ and invertible $m-l \times m-l$ matrix.
With this input the right inverses have the structure
$$
\hat A = \pmat{A_0^{-1} \\ \hat A_1}\;,\quad \hat B=\pmat{B_{00}^{-1} & \nul \\ -B_{11}^{-1} B_{10} B_{00}^{-1} & B_{11}^{-1} \\ \hat B_{02} & \hat B_{12}}\;,\quad
\hat C=\pmat{C_0^{-1} \\ \hat C_1 \\ \hat C_2}
$$
where $\hat A_1$, $\hat B_{02}$, $\hat B_{12}$, $\hat C_1$ and $\hat C_2$ are arbitrary matrices of the corresponding sizes.
Now we obtain
$$
\hat B \hat A\,=\, \pmat{C_0^{-1} \\ -B_11^{-1} B_{10} C_0^{-1} +B_{11}^{-1} \hat A_1 \\ \hat B_{02} A_0^{-1}+\hat B_{12} \hat A_1}
$$
Here we used $C_0^{-1}=B_{00}^{-1} A_0^{-1}$ which follows from the observations above. So obviously $\hat B \hat A$ is a right inverse to $C$ for any choice of $\hat A$ and $\hat B$.
On the other hand, for any $\hat C_1$, $\hat C_2$ we can elect $\hat B_{12}=\nul$ and solve for $\hat A_1$ and $\hat B_{02}$ to obtain
$$
\hat C_1\,=\,-B_11^{-1} B_{10} C_0^{-1} +B_{11}^{-1} \hat A_1 \qtx{and}
\hat B_{02} A_0^{-1}\,=\,\hat C_2\;.
$$
\end{proof}

\section{Existence of absolutely continuous component in weak limits of measures}

In Theorem~\ref{Theo-1a}~(ii) we use the following criterion which follows from \cite[Lemma~2]{DK}.

\begin{prop}\label{prop-B1}
Let $\Omega$ be a open subset of $\RR$ and $w(\lambda)\in L^1_{loc}(\Omega)$ which is Lebesgue almost everywhere positive. Let $\mu_n$ be a set of positive measures which converge weakly to $\mu$ and let $\frac{d\mu_n}{d\lambda}$ denote the Radon Nikodym derivative with respect to the Lebesgue measure.
Furthermore let $K\subset \Omega$ be a compact set of positive Lebesgue measure and we denote
$w(K):=\int_K w(\lambda) d\lambda$. If we have
$$
\liminf_{n\to \infty} \frac{1}{w(K)} \int_K - \log\left( \frac{d\mu_n}{d\lambda} \,\frac{1}{w(\lambda}\right)\,w(\lambda)\,d\lambda\,<\, \infty
$$
then $\mu(K)>0$ . 
In particular, if we have this property for any compact set $K\subset \Omega$ of positive Lebesgue measure, then the measure $\mu$ contains an absolutely continuous part in all of $\Omega$, i.e.
$\Omega \subset \supp (\mu_{ac})$, where $\mu_{ac}$ is the absolutely continuous part of $\mu$.
\end{prop}

This follows directly from the following lemma.
\begin{lemma}[Lemma~2 of \cite{DK}]
Under the assumptions as in Proposition~\ref{prop-B1} we have
$$
\liminf_{n\to \infty} \frac{1}{w(K)} \int_K - \log\left( \frac{d\mu_n}{d\lambda} \,\frac{1}{w(\lambda}\right)\,w(\lambda)\,d\lambda\,\geq \log\left( \frac{w(K)}{\mu(K)}\right)
$$
where $w(K):=\int_K w(\lambda) d\lambda > 0$.
\end{lemma}
\begin{proof}
Let $\Phi$ denote the set of positive continuous functions of compact support which take value 1 on $K$, then
\begin{align*}
\mu(K) = \inf_{\phi\in\Phi} \int \phi d\mu &= \inf_{\phi\in\Phi} \lim_{n\to\infty} \int \phi d\mu_n \geq
\inf_{\phi\in\Phi} \limsup_{n\to\infty} \int \phi \frac{d\mu_n(\lambda)}{d\lambda}\,d\lambda\\
 &\geq \limsup_{n\to\infty} \int_K \frac{d\mu_n(\lambda)}{d\lambda}\,\frac{1}{w(\lambda)} \;w(\lambda)
\end{align*}
Now divide by $w(K)>0$  take   $-\log(\cdot)$ on both sides and use Jensen's  inequality to get the result. Note that taking $-\log(\cdot)$ changes the inequality sign and $\limsup$ to $\liminf$.
\end{proof}

\end{document}